\RequirePackage{fix-cm}
\documentclass{svjour3}                     %
\smartqed  %
\usepackage{graphicx}
\usepackage{listings}
\usepackage{epsfig}
\usepackage{lscape}
\usepackage{multirow}
\usepackage{longtable}
\usepackage{amsmath,amssymb}
\usepackage{color}
\usepackage{placeins}
\usepackage{url}
\usepackage{cases}
\usepackage{hyperref}
\usepackage{algorithm}  
\usepackage{setspace}
\usepackage{subfigure}
\usepackage{pdflscape}
\usepackage{algpseudocode}
\usepackage{booktabs}

\DeclareMathOperator*{\argmin}{argmin}

\newtheorem{assumption}{Assumption}

\begin{document}

\title{\bf A dynamic programming approach for generalized  nearly isotonic  optimization}
\titlerunning{DP for GNIO}        %

\author{Zhensheng Yu  \and
		Xuyu Chen \and
		Xudong Li 
}

\institute{Zhensheng Yu \at College of Science, University of Shanghai for Science and Technology, No. 516, Jungong Road, Shanghai, China \\
               \email{zhsh-yu@163.com}
           \and
           Xuyu Chen \at School of Mathematics, Fudan University, No. 220, Handan Road, Shanghai, China \\
               \email{chenxy18@fudan.edu.cn}
           \and
           Xudong Li \at School of Data Science, Fudan University, No. 220, Handan Road, Shanghai, China \\
           		 \email{lixudong@fudan.edu.cn}
}

\date{Received: date / Accepted: date}

\maketitle

\begin{abstract}
Shape restricted statistical estimation problems have been extensively studied, with many important practical applications in signal processing, bioinformatics, and machine learning. In this paper, we propose and study a generalized nearly isotonic optimization (GNIO) model, which recovers, as special cases, many classic problems in shape constrained statistical regression, such as   isotonic regression, nearly isotonic regression and  unimodal regression problems. We develop an efficient and easy-to-implement dynamic programming algorithm for solving the proposed model whose recursion nature is carefully uncovered and exploited. For special $\ell_2$-GNIO problems, implementation details and the optimal ${\cal O}(n)$ running time analysis of our algorithm are discussed. Numerical experiments, including the comparisons among our approach, the powerful commercial solver Gurobi, and existing fast algorithms for solving $\ell_1$-GNIO and $\ell_2$-GNIO problems, on both simulated and real data sets, are presented to demonstrate the high
	efficiency and robustness of our proposed algorithm in solving large scale GNIO problems. 
\end{abstract}

\vfill

\section{Introduction}
\label{intro}
In this paper, we are interested in solving the following convex composite optimization problem:
\begin{equation} \label{eq:gip}
\min_{x\in\Re^n} \quad \sum_{i=1}^{n} f_i(x_i) + \sum_{i=1}^{n-1} \lambda_{i} (x_i - x_{i+1})_{+} + \sum_{i=1}^{n-1} \mu_{i} (x_{i+1} - x_{i})_{+},
\end{equation}
where  $f_i:\Re \to \Re$, $i=1,\ldots,n$,  are convex loss functions, $\lambda_i$ and $\mu_i$ are nonnegative and possibly {infinite} scalars, i.e.,
\[
0\le \lambda_i, \mu_i \le +\infty, \quad \forall\, i = 1,\ldots, n,
\] 
and $(x)_{+} = \max(0,x)$ denotes the nonnegative part of $x$  for any $x\in\Re$. Here, if for some $i\in \{1,\ldots, n-1\}$, $\lambda_{i} = +\infty$ (respectively, $\mu_{i} = +\infty$), the corresponding regularization term $\lambda_{i}(x_{i} - x_{i+1})_+$ (respectively,  $\mu_{i}(x_{i+1} - x_{i})_+$) in the objective of \eqref{eq:gip} shall be understood as the indicator function $\delta(x_i, x_{i+1} \mid x_i - x_{i+1} \le 0)$ (respectively, $\delta(x_i, x_{i+1} \mid x_{i} -  x_{i+1} \ge 0)$), or equivalently the constraint $x_i - x_{i+1} \le 0$ (respectively, $x_{i} -  x_{i+1} \ge 0$).
To guarantee the existence of optimal solutions to problem \eqref{eq:gip}, throughout the paper, we make the following blanket assumption:
\begin{assumption}\label{assum:level_set}
	Each $f_i:\Re \to \Re$, $i=1,\ldots,n$, is a convex function and has bounded level-sets, i.e., there exists $\alpha \in \Re$ such that the $\alpha$-level set $\left\{x\in\Re\mid f_i(x)\le \alpha \right\}$ is non-empty and bounded. 
\end{assumption}
\noindent Indeed, under this mild assumption, it is not difficult to see that the objective function in problem \eqref{eq:gip} is also a convex function and has bounded level-sets. Then, by \cite[Theorems 27.1 and 27.2]{rockafellar1970convex}, we know that problem \eqref{eq:gip} has a non-empty optimal solution set.

Problem \eqref{eq:gip} is a generalization of the following nearly isotonic regression problem \cite{Tibshirani2016nearly}:
\begin{equation}
\label{prob:niso}
\min_{x\in \Re^n} \;  \frac{1}{2} \sum_{i=1}^{n} (x_i - y_i)^2 + \lambda \sum_{i=1}^{n-1}(x_i - x_{i+1})_{+},
\end{equation}
where the nonnegative scalar $\lambda \in\Re$ is a given parameter, $y_i\in\Re$, $i=1,\ldots, n$, are $n$ given data points. Clearly, \eqref{eq:gip} can be regarded as a generalization of model \eqref{prob:niso} in the sense that more general convex loss functions $f_i$ and regularizers (constraints) are considered; hence we refer \eqref{eq:gip} as the {\it generalized nearly isotonic optimization} ({\bf GNIO}) problem and the term  $\sum_{i=1}^{n-1} \lambda_{i} (x_i - x_{i+1})_{+} + \sum_{i=1}^{n-1} \mu_{i} (x_{i+1} - x_{i})_{+}$ in the objective of \eqref{eq:gip} as the {\it generalized nearly isotonic regularizer}. We further note that model \eqref{eq:gip} also subsumes the following fused lasso problem \cite{friedman2007pathwise}, or $\ell_2$ total variation regularization problem:
\begin{equation}
\label{prob:flsa}
\min_{x\in \Re^n}  \; \frac{1}{2} \sum_{i=1}^{n} (x_i - y_i)^2 + \lambda \sum_{i=1}^{n-1}|x_i - x_{i+1}|.
\end{equation}
By allowing the parameters $\lambda_i$ and/or $\mu_i$ in  \eqref{eq:gip}  to be positive infinite, we see that the generalized nearly isotonic optimization model \eqref{eq:gip} is closely related to {the} shape constrained inference problems \cite{Silvapulle2005constrained}. 
For instance, the important isotonic regression \cite{Ayer1955emprical,Bartholomew1959bio,bartholomew1959test,Bartholomew1972statistic,Brunk1955maximum,Ahuja2001fast} and  unimodal regression problems \cite{Quentin2008unimodal,Frisen1986unimodal} are special cases of model \eqref{eq:gip}. Indeed, given a positive weight vector $w\in\Re^n$ and some $1\le p < +\infty$, in problem \eqref{eq:gip}, if by taking 
\[ f_i(x_i) = w_i|x_i - y_i|^p, \quad i=1,\ldots,n,\]
 and $\lambda_i = +\infty$ and $\mu_i = 0$ for $i=1,\ldots, n-1$, we obtain the following $\ell_p$ isotonic regression problem:
\begin{equation}
\label{prob:iso_reg}
\begin{aligned}
& \min_{x\in \Re^n} \quad  \sum_{i=1}^{n} w_i|x_i - y_i|^p  \\
& \quad \text{s.t.}  \quad x_i \le x_{i+1}, \quad i = 1, \cdots,n-1.
\end{aligned} 
\end{equation}
In the above setting, if for some $1 \le m \le n-2$, by taking 
\begin{align*}
\lambda_{i} = +\infty, \, \mu_i = 0, \quad i=1,\ldots,m, \quad \mbox{ and } \, \lambda_{i} =0, \, \mu_i = +\infty, \, i=m+1,\ldots,n-1,
\end{align*}
in model \eqref{eq:gip},
we obtain the following unimodal regression problem:
\begin{equation}
\label{prob:uni_reg}
\begin{aligned}
\min_{x\in \Re^n} &\quad \sum_{i=1}^{n} w_i|x_i - y_i|^p  \\
\quad \text{s.t.}  &\quad x_i \le x_{i+1}, \quad i = 1,\cdots,m, \\
&\quad x_i \ge x_{i+1},  \quad i = m+1,\cdots,n-1.
\end{aligned} 
\end{equation}
These related problems indicate wide applicability of our model \eqref{eq:gip} in many fields including operations research \cite{Ahuja2001fast}, signal processing \cite{xiaojun2016videocut,Alfred1993localiso},   medical prognosis \cite{Ryu2004medical},  and traffic and climate data analysis \cite{Matyasovszky2013climate,Wu2015traffic}.
Perhaps the closest model to our generalized isotonic optimization problem is the Generalized Isotonic Median Regression (GIMR) model studied in \cite{Hochbaum2017faster}. While we allow loss functions $f_i$ in problem \eqref{eq:gip} to be general convex functions, the GIMR model in \cite{Hochbaum2017faster} assumes that each $f_i$ is piecewise affine. Besides, our model \eqref{eq:gip} deals with the positive infinite parameters $\lambda_i$ and/or $\mu_i$ more explicitly. We further note that in model \eqref{eq:gip} we do not consider box constraints for each decision variable $x_i$ as in the GIMR, since many important instances of model \eqref{eq:gip}, e.g., problems \eqref{prob:niso}, \eqref{prob:flsa}, \eqref{prob:iso_reg} and \eqref{prob:uni_reg}, contain no box constraints. Nevertheless, as one can observe later, our analysis can be generalized to the case where additional box constraints of $x_i$ are presented without much difficulty.

Now we present a brief review on available algorithms for solving the aforementioned related models. 
For various shape constrained statistical regression problems including problems \eqref{prob:iso_reg} and \eqref{prob:uni_reg}, 
a widely used and efficient algorithm is the pool adjacent violators algorithm (PAVA) \cite{Ayer1955emprical,Bartholomew1972statistic}. The PAVA was originally proposed for solving the $\ell_2$ isotonic regression problem, i.e., problem \eqref{prob:iso_reg} with $p = 2$. In \cite{Best1990active}, Best and Chakravarti proved that the PAVA, when applied to the $\ell_2$ isotonic regression problem, is in fact a dual feasible active set method. Later, the PAVA was further extended to handle isotonic regression problems with general separable convex objective in \cite{Stromberg1991algorithm,Best2000minimizing,Ahuja2001fast}. Moreover, the PAVA was generalized to solve the $\ell_p$ unimodal regression \cite{Quentin2008unimodal} with emphasis on the $\ell_1$, $\ell_2$ and $\ell_\infty$ cases. Recently, Yu and Xing \cite{Yu2016exact} proposed a generalized PAVA for solving separable convex minimization problems with rooted tree order constraints. Note that the PAVA can also be generalized to handle the nearly isotonic regularizer. In \cite{Tibshirani2016nearly}, Tibshirani et al. developed an algorithm, which can be viewed as a modified version of the PAVA, for computing the solution path of the nearly isotonic regression problem \eqref{prob:niso}.  Closely related to the PAVA, a direct algorithm for solving $\ell_2$ total variation regularization \eqref{prob:flsa} was proposed in \cite{Condat2013direct} by Condat, which appears to be {one of the fastest algorithms} for solving \eqref{prob:flsa}. Despite the wide applicability of the PAVA, it remains largely unknown whether the algorithm can be modified to solve the general GNIO problem \eqref{eq:gip}.

Viewed as a special case of the Markov Random Fields (MRF) problem \cite{Hochbaum2001efficient}, the GIMR model \cite{Hochbaum2017faster} was efficiently solved by emphasizing the piecewise affine structures of the loss functions $f_i$ and carefully adopting the cut-derived threshold theorem of Hochbaum's MRF algorithm \cite[Theorem 3.1]{Hochbaum2001efficient}. However, it is not clear
whether the proposed algorithm in \cite{Hochbaum2017faster} can be extended to handle general convex loss functions in
the GNIO problem \eqref{eq:gip}. Quite recently, as a generalization of the GIMR model, Lu and Hochbaum \cite{Hochbaum2021unified} investigated the 1D generalized total variation problem where in the presence of the box constraints over the decision variables, general real-valued convex loss and regularization functions are considered. Moreover, based on the Karush-Kuhn-Tucker optimality conditions of the 1D generalized total variation problem, an efficient algorithm with a nice complexity was proposed in \cite{Hochbaum2021unified}. However, special attentions are needed  when the algorithm is applied to handle problems consisting extended real-valued regularization functions (e.g., indicator functions), as the subgradient of the regularizer at certain intermediate points of the algorithm may be empty. Hence, it may not be able to properly handle problem \eqref{eq:gip} when some parameters $\lambda_i$ and/or $\mu_i$ are positive infinite.

 On the other hand, based on the inherit recursion structures of the underlying problems, efficient dynamic programming (DP) approaches \cite{Rote2019isoDP,Johnson2013dplasso} are designed to solve the regularized problem \eqref{prob:flsa} and the constrained problem \eqref{prob:iso_reg} with $p=1$. Inspired by these successes, we ask the following question:
 
 {\it Can we design an efficient DP based algorithm to handle more sophisticated shape restricted statistical regression models than previously considered problems \eqref{prob:flsa} and \eqref{prob:iso_reg} with $p=1$?}
 
In this paper, we provide an affirmative answer to this question. In fact, our careful examination of the algorithms in \cite{Rote2019isoDP,Johnson2013dplasso} reveals great potential of the dynamic programming approach for handling general convex loss functions and various order restrictions as regularizers and/or constraints which are  emphasized in model \eqref{eq:gip}. 
Particularly, we propose the first efficient and implementable dynamic programming algorithm for solving the general model \eqref{eq:gip}. 
By digging into the recursion nature of \eqref{eq:gip}, we start by reformulating problem \eqref{eq:gip} into an equivalent form consisting a series of recursive optimization subproblems, which are suitable for the design of a dynamic programming approach.
 Unfortunately, the involved objective functions are also recursively defined and their definitions require solving infinitely many optimization problems\footnote{See \eqref{eq:hxi} for more details.}. Therefore, the naive extension of the dynamic programming approaches in \cite{Rote2019isoDP,Johnson2013dplasso} will not result a computationally tractable algorithm for solving \eqref{eq:gip}.
Here, we overcome this difficulty by utilizing the special properties of the generalized nearly isotonic regularizer in \eqref{eq:gip} to provide explicit updating formulas for the objective functions, as well as optimal solutions, of the involved subproblems. Moreover, the computations associated with each of the aforementioned formulas only involve solving at most two univariate convex optimization problems. These formulas also lead to a semi-closed formula, in a recursive fashion, of an optimal solution to \eqref{eq:gip}.
As an illustration, the implementation details and the optimal ${\cal O}(n)$ running time of our algorithm for solving $\ell_2$-GNIO problems\footnote{Its definition can be found at the beginning of Section \ref{sec:imple_quadratic}.} are discussed. We also conduct extensive numerical experiments to  demonstrate the robustness, as well as effectiveness, of our algorithm for handling different generalized nearly isotonic optimization problems.

The remaining parts of this paper are organized as follows. In the next section, we provide some discussions on the subdifferential mappings of univariate convex functions. The obtained results will be used later for designing and analyzing our algorithm. In Section 3, a dynamic programming approach is proposed for solving problem \eqref{eq:gip}. The explicit updating rules for objectives and optimal solutions of the involved optimization subproblems are also derived in this section.
In Section 4, we discuss some practical implementation issues and conduct running time analysis of our algorithm for solving $\ell_2$-GNIO problems. Numerical experiments are presented in Section 5. We conclude our paper in the final section.

\section{Preliminaries}
\label{sec:preliminary}
In this section, we discuss some basic properties associated with the subdifferential mappings of univariate convex functions. These properties will be extensively used in our algorithmic design and analysis.

Let $f:\Re \to \Re$ be any given function. We define the corresponding left derivative $f'_-$ and right derivative $f'_+$ at a given point $x$ in the following way:
\begin{equation*}
\label{def:lfrf}
f'_-(x) = \lim_{z\,\uparrow\,x} \frac{f(z) - f(x)}{z - x}, \quad 
f'_+(x) = \lim_{z\,\downarrow\,x} \frac{f(z) - f(x)}{z - x},
\end{equation*}
if the limits exist.
Of course, if $f$ is actually differentiable at $x$, then $f'_-(x) = f'_+(x) = f'(x)$. If $f$ is assumed to be a convex function, one can say more about $f'_+$ and $f'_-$ and their relations with the subgradient mapping $\partial f$. We summarize these properties in the following lemma which are mainly taken from \cite[Theorems 23.1, 24.1]{rockafellar1970convex}. Hence, the proofs are omitted here.
\begin{lemma}
	\label{lem:lfrf}
	Let $f:\Re \to \Re$ be a convex function. Then, $f'_+$ and $f'_-$ are well-defined finite-valued non-decreasing functions on $\Re$,  such that
	\begin{equation}
	\label{eq:nondlfrf}
	f_+'(z_1) \le f_-'(x)\le f_+'(x)\le f_-'(z_2) \quad \mbox{ whenever }\quad z_1 < x < z_2.
	\end{equation}
	For every $x$, one has
	\begin{subequations}
	\begin{align}
	& f_-'(x) = \sup_{z < x} \frac{f(z) - f(x)}{z - x}, \quad f_+'(x) = \inf_{z > x} \frac{f(z) - f(x)}{z - x}, \label{eq:lfrfa}
	\\[2pt]
	& f_-'(x) = \lim_{z \, \uparrow \, x} f_-'(z)= \lim_{z \, \uparrow \, x} f'_+(z), \quad f_+'(x) = \lim_{z \, \downarrow \, x} f'_+(z) = \lim_{z \, \downarrow \, x} f'_-(z). \label{eq:lfrfb}
	\end{align}
	\end{subequations}
	In addition, it holds that 
	\begin{equation}
	\label{eq:pg}
	\partial f(x) = \left\{d\in\Re \mid f'_-(x)\le d \le f'_+(x) \right\}, \quad \forall \, x\in\Re.
	\end{equation}
\end{lemma}

Given a convex function $f:\Re \to \Re$, the left or right continuity of $f'_-$ and $f'_+$ derived in Lemma \ref{lem:lfrf}, together with their non-decreasing properties, imply that certain intermediate value {theorem} holds for the subgradient mapping $\partial f$. 
\begin{lemma}

	\label{lem:intermediate}
	Let $f:\Re \to \Re$ be a convex function. 
	Given an interval $[x_1,x_2]\subseteq \Re$, for any $\alpha \in [f'_-(x_1), f'_+(x_2)]$, there exists $c\in \Re$ such that $f_-'(c)\le \alpha \le f'_+(c)$, i.e., $\alpha \in \partial f(c)$. In fact, for the given $\alpha$, a particular choice is
	\[ c = \sup\left\{ z\in [x_1,x_2]\mid f'_-(z) \le \alpha \right\}. \] 
\end{lemma}

\begin{proof}
	Let $S = \left\{ z\in [x_1,x_2]\mid f'_-(z) \le \alpha \right\}$.
	Since $\alpha \ge f'_-(x_1)$, we know that $x_1 \in S$, i.e., $S$ is nonempty. Meanwhile, since $S$ is also upper bounded, by the completeness of the real numbers, we know that $c = \sup S$ exists and $x_1 \le c \le x_2$. 
	
	Assume on the contrary that $\alpha \not \in \partial f(c)$. Then, we can construct an operator $T$ through its graph 
	\[
	{\rm graph}\, T = {\rm graph}\, \partial f	\cup \{(c,\alpha)\}.
	\]
	It is not difficult to verify that $T$ is monotone.  Since $\partial f$ is maximal monotone \cite{rockafellar1970convex}, it holds that $T = \partial f$. We arrive at a contradiction and complete the proof of the lemma.
\end{proof}
\begin{corollary}\label{cor:intermediate}
Let $f:\Re \to \Re$ be a convex function and suppose that $\inf_x f'_-(x) < \sup_x f'_+(x)$. For any given $\alpha \in (\inf_x f'_-(x),\, \sup_x f'_+(x) )$, there exists $c\in \Re$ such that $\alpha \in \partial f(c)$ and a particular choice is \[c = \sup\left\{ z\in \Re\mid f'_-(z) \le \alpha \right\}.\]
\end{corollary}
\begin{proof}
	Since $\inf_x f'_-(x) < \alpha < \sup_x f'_+(x)$, there exist $x_1, x_2\in \Re$ such that $x_1 < x_2$ and $ -\infty < f'_-(x_1) < \alpha < f'_+(x_2) < +\infty$. The desired result then follows from  the fact that
	\[\sup\left\{ z\in \Re\mid f'_-(z) \le \alpha \right\} = \sup \left\{ z\in [x_1,x_2]\mid f'_-(z) \le \alpha \right\}\] and Lemma \ref{lem:intermediate}.
\end{proof}

\begin{proposition}
	\label{prop:optset}
Let $f:\Re \to \Re$ be a convex function and suppose that $\inf_x f'_-(x) < \sup_x f'_+(x)$. Given $\lambda\in \Re$, let $\Omega(\lambda)$ be the optimal solution set to the following problem:
\[
\min_{x\in\Re} \displaystyle  \left\{ f(x) - \lambda x \right\} .
\]
Then, $\Omega(\lambda)$ is nonempty if and only if either one of the following two conditions holds:
\begin{enumerate}
	\item there exists $\bar x\in \Re$ such that $f_-'(\bar x) = \inf_x f'_-(x) = \lambda$ or $f_+'(\bar x) = \sup_x f'_+(x) = \lambda$;
	\item $\inf_x f'_-(x) < \lambda < \sup_x f'_+(x)$.
\end{enumerate}
\end{proposition}

\begin{proof}
Observe that $\Omega(\lambda) = \left\{x\in \Re \mid \lambda\in \partial f(x) \right\} = \left\{x\in\Re\mid f_-'(x) \le \lambda \le f_+'(x) \right\}$. Thus, we see that if none of the above two conditions holds, then $\Omega(\lambda)$ is an empty set, i.e., the ``only if'' part is proved.	

Next, we focus on the ``if'' part. Suppose that condition 1 holds. Then, $\lambda \in \partial f(\bar x)$, i.e., $\bar x \in \Omega(\lambda)$. Meanwhile, if condition 2 holds, by Corollary \ref{cor:intermediate}, we know that there exists $c\in \Re$ such that $\lambda \in \partial f(c)$, i.e., $c\in \Omega(\lambda)$. Therefore, in both cases, we know that $\Omega(\lambda) \neq  \emptyset$. We thus complete the proof. 
\end{proof}

Let $f:\Re \to \Re$ be a convex function. Denote by $f0^+$ the recession function of $f$. We summarize in the following lemma some useful properties associated with $f0^+$. The proofs can be founded in \cite[Propositions 3.2.1, 3.2.4, 3.2.8]{hiriart2004fundamentals}.

\begin{lemma}\label{lem:recessf}
	Let $f,g$ be two real-valued convex functions over $\Re$. 
	It holds that 
	\begin{itemize}
		\item 	for any $d\in\Re$, $(f0^+)(d) = \lim_{t\to +\infty} (f(x_0 + t d) - f(x_0))/t,$
		where $x_0 \in \Re$ is arbitrary;
		\item $f$ has bounded level-sets if and only if $(f0^+)(d) > 0$ for all $d\neq 0$;
		\item $(f+g)0^+ = f0^+ + g0^+$.
	\end{itemize}
\end{lemma}

\section{A dynamic programming algorithm for solving (\ref{eq:gip})}
\label{sec3:dynamic}
In this section, we shall develop a dynamic programming algorithm for solving the generalized nearly isotonic optimization problem \eqref{eq:gip}. Inspired by similar ideas explored in \cite{Johnson2013dplasso,Rote2019isoDP}, we uncover the recursion nature of problem \eqref{eq:gip} and reformulate it into a form which is suitable for dynamic programming approaches.

Let $h_1(x_1) = 0$ for all $x_1\in\Re$ and for $i=2,\ldots, n$,  define recursively functions $h_i$ by
\begin{equation}
\label{eq:hxi}
 h_i(x_i) 
:=  \min_{x_{i-1} \in \Re} \left\{ \begin{aligned}
& f_{i-1}(x_{i-1}) + h_{i-1}(x_{i-1})  \\ 
& + \lambda_{i-1}(x_{i-1} - x_i)_+ + \mu_{i-1}(x_i - x_{i-1})_+
\end{aligned} 
\right\}, \, \forall \, x_i \in \Re.
\end{equation}
Then, it holds for any $2\le i \le n$  that
\begin{equation}
\label{eq:fihi}
\begin{aligned}
 &\min_{x_i \in \Re} \left\{ f_i(x_i) + h_i(x_i) \right\} \\[5pt] 
= & \min_{x_i\in\Re,\ldots,x_1\in\Re} \left\{\sum_{j=1}^{i} f_j(x_j) + \sum_{j=1}^{i-1}\lambda_j (x_{j} - x_{j+1})_+ + \sum_{j=1}^{i-1} \mu_j (x_{j+1} - x_{j})_+ \right\}.
\end{aligned}
\end{equation}
In particular, the optimal value of problem \eqref{eq:gip} can be obtained in the following way:
\begin{align*}
&\min_{x_1\in \Re,\cdots,x_n\in\Re} \left\{ \sum_{i=1}^{n} f_i(x_i) + \sum_{i=1}^{n-1} \lambda_{i} (x_i - x_{i+1})_{+} + \sum_{i=1}^{n-1} \mu_{i} (x_{i+1} - x_{i})_{+} \right\} \\
= {}&\min_{x_n \in \Re} \left\{ f_n(x_n) + h_n(x_n) \right\}.
\end{align*}
These observations allow us to solve problem \eqref{eq:gip} via solving a series of subproblems involving functions $h_i$.
Indeed, suppose that \[x_n^* \in \argmin_{x_n\in\Re}  \left\{ f_n(x_n) + h_n(x_n) \right\}.\]
For $i = n-1, \ldots, 1$, let $x^*_i$ be recursively defined as an optimal solution to the following problem:
\begin{equation}
\label{eq:backward_xi}
x_i^* \in \argmin_{x_i\in\Re} \left\{
f_i(x_i) + h_i(x_i) + \lambda_i (x_i - x_{i+1}^*)_+ + \mu_i (x_{i+1}^* - x_i)_+
\right\}.
\end{equation}
Based on \eqref{eq:fihi}, one can easily show that $x^* = (x_1^*,\ldots, x_n^*) \in \Re^n$ solves problem \eqref{eq:gip}. For later use, we further define $g_i = f_i + h_i$ for all $i=1,\ldots, n$.

The above observations inspire us to apply the following dynamic programming algorithm for solving problem \eqref{eq:gip}. To express the high-level idea more clearly, we only present the algorithm in the most abstract form here. More details will be revealed in  subsequent discussions.

\begin{algorithm}[H]
\caption{A dynamic programming algorithm for problem \eqref{eq:gip} }
\label{alg:DP}
\begin{algorithmic}[1]
\State {Initialize: $h_1(x) = 0$}
\For {$i = 2: n$}
\State {$g_{i-1} = \text{\bf sum}(h_{i-1},f_{i-1})$}
\State {$(h_{i},b_i^-,b_i^+) = \text{\bf update}(g_{i-1}, \lambda_{i-1}, \mu_{i-1})$}
\EndFor
\State {$g_n = \text{\bf sum}(h_n,f_n)$}
\State {$x_n = \argmin_{x\in\Re} g_n(x)$}
\State {$x_1,\cdots,x_{n-1} = \text{\bf recover}( x_n$,$\{ (b_i^-,b_i^+) \}_{i=2}^{n})$ }
\State {\textbf{Return} $x_1,\cdots,x_{n-1},x_n$ }
\end{algorithmic}
\end{algorithm}

In the $i$-th iteration of the for-loop in the above algorithm, the ``sum'' function computes the summation of $f_{i-1}$ and $h_{i-1}$ to obtain $g_{i-1}$, i.e., $g_{i-1} = h_{i-1} + f_{i-1}$. Based on the definition of $h_i$ in \eqref{eq:hxi}, the ``update'' function computes $h_i$ from $g_{i-1}$. The extra outputs $b_i^-$ and $b_i^+$ from this step will be used to recover the optimal solution in the ``recover'' function which is based on the backward computations in \eqref{eq:backward_xi}. 
Hence, both the ``update'' and the ``recover'' steps involve solving similar optimization problems in the form of \eqref{eq:hxi}. A first glance of the definition of $h_i$ in \eqref{eq:hxi} may lead us to the conclusion that the ``update'' step is intractable. The reason is that one has to compute $h_i$ over all $x_i\in \Re$, i.e., infinitely many optimization problems have to be solved. To alleviate this difficulty, based on a careful exploitation of the special structures of the generalized nearly isotonic regularizer in \eqref{eq:hxi}, we present an explicit formula to compute $h_i$. Specifically, we are able to 
determine $h_i$ by calculating two special breakpoints $b_i^-, b_i^+ \in \Re$ via solving at most two one-dimensional optimization problems. 
Moreover, these breakpoints will also be used in the ``recover'' step. In fact, as one will observe later, an optimal solution to problem \eqref{eq:backward_xi} enjoys a closed-form representation involving $b_i^-$ and $b_i^+$.   
A concrete example on the detailed implementations of these steps will be discussed in Section \ref{sec:imple_quadratic}.

\subsection{An explicit updating formula for ($\ref{eq:hxi}$)}
In this subsection, we study the ``update'' step in Algorithm \ref{alg:DP}. In particular, we show how to obtain an explicit updating formula for $h_i$ defined in \eqref{eq:hxi}. 

We start with a more abstract reformulation of \eqref{eq:hxi}. Given a univariate convex function $g:\Re \to \Re$ and nonnegative and possibly infinite constants $\mu,\lambda$, for any given $y\in \Re$, let 
\begin{equation} \label{def:zS}
p_y(x) = \lambda(x-y)_+ + \mu(y-x)_+ \; \mbox{ and } \;
z_y(x) = g(x) + p_y(x), \quad x\in\Re.
\end{equation}
Here, if for some, $\lambda = +\infty$ (respectively, $\mu = +\infty$), the corresponding regularization term $\lambda(x - y)_+$ (respectively,  $\mu(y - x)_+$) in $p_y$ shall be understood as the indicator function $\delta(x\mid x\le y)$ (respectively, $\delta(x\mid x\ge y)$).
We focus on the optimal value function $h$ defined as follows:
\begin{equation}\label{eq:h}
h(y) := \min_{x\in\Re} z_y(x), \quad y\in \Re. 
\end{equation}
For the well-definedness of $h$ in \eqref{eq:h}, we further assume that $g$ has bounded level-sets, i.e., there exists $\alpha \in \Re$ such that the set $\left\{x\in\Re\mid g(x)\le \alpha \right\}$ is non-empty and bounded. Indeed, under this assumption, it is not difficult to see that $z_y$ is also a convex function and has bounded level-sets. Then, by \cite[Theorems 27.1 and 27.2]{rockafellar1970convex}, we know that
for any $y$, problem $\min_x z_y(x)$ has a non-empty and bounded optimal solution set. Therefore, the optimal value function $h$ is well-defined on $\Re$ with ${\rm dom}\, h = \Re$.

 Since ${\rm dom} \, g = \Re$, we know from the definitions of $z_y$ and $p_y$ in \eqref{def:zS} and \cite[Theorem 23.8]{rockafellar1970convex} that 
\begin{equation}
\label{eq:subggz}
\partial z_y(x) = \partial g(x) + \partial p_y(x), \quad \forall\, x\in\Re.
\end{equation}
Since $g$ has bounded level-sets, it holds from \cite[Theorems 27.1 and 27.2]{rockafellar1970convex} that the optimal solution set to $\min_x g(x)$, i.e., $S = \{u\in\Re\mid 0\in \partial g(u)\}$, is nonempty and bounded. Let $u^*\in S$ be an optimal solution, then, by Lemma \ref{lem:lfrf}, $\inf_x g'_-(x) \le g'_-(u^*)\le 0 \le g'_+(u^*) \le \sup_x g'_+(x)$. We further argue that 
\begin{equation}
\label{eq:gm0gp}
\inf_x g'_-(x) < 0 < \sup_x g'_+(x).
\end{equation} Indeed, if $\inf_x g'_-(x) = 0$, then \[0 = \inf_x g'_-(x) \le g'_-(x) \le g'_-(u^*) \le 0, \quad \forall\,  x\in (-\infty, u^*],\]
i.e., $g'_-(x) = 0$ for all $ x\in (-\infty, u^*]$. Therefore, by Lemma \ref{lem:lfrf}, we know that $0 \in \partial g(x)$ for all $ x\in (-\infty, u^*]$, i.e., $(-\infty, u^*] \subseteq S$. This contradicts to the fact that $S$ is bounded. Hence, it holds that $ \inf_x g'_-(x) < 0$. Similarly, one can show that $\sup_x g'_+(x) > 0$. 

Next, based on observation \eqref{eq:gm0gp}, for any given nonnegative and possibly infinite parameters $\lambda, \mu$, it holds that
$ -\lambda \le 0 < \sup_x g'_+(x)$ and $\mu \ge 0 > \inf_x g'_-(x)$. Now, we define two breakpoints $b^-$ and $b^+$ associated with the function $g$ and parameters $\lambda, \mu$. 
Particularly, we note from \cite[Theorem 23.5]{rockafellar1970convex} that 
\begin{equation}
\label{eq:omega}
\begin{aligned}
\partial g^*(\alpha) = {}& \left\{ x\in \Re\mid 
\alpha \in \partial g(x)
\right\} \\
 ={}& \mbox{the optimal solution set of problem } \min_{z\in\Re}\; \{ g(z) - \alpha z \}.
\end{aligned}
\end{equation}
Define
\begin{equation}
\label{eq:bm}
b^-  \begin{cases}
 = -\infty, \; & \mbox{ if } \lambda = +\infty \mbox{ or } \partial g^*(-\lambda) =  \emptyset, \\[2pt]
 \in \partial g^*(-\lambda), \; &\mbox{ otherwise}, 
\end{cases} 
\end{equation}
and
\begin{equation}
\label{eq:bp}
b^+  \begin{cases}
= +\infty, \; & \mbox{ if } \mu = +\infty \mbox{ or } \partial g^*(\mu) =  \emptyset, \\[2pt]
\in \partial g^*(\mu), \; &\mbox{ otherwise}
\end{cases} 
\end{equation}
with the following special case
\begin{equation}
\label{eq:bpm0}
b ^+ = b^- \in \partial g^*(0), \quad \mbox{if } \lambda = \mu = 0.
\end{equation}
Here, the nonemptiness of $\partial g^*(0)$ follows from \eqref{eq:gm0gp} and Proposition \ref{prop:optset}. {In fact, Proposition \ref{prop:optset} and \eqref{eq:gm0gp} guarantee that there exist parameters $\lambda$ and $\mu$ such that $b^-$ and $b^+$ are finite real numbers.}
Moreover, as one can observe from the above definitions, to determine $b^-$ and $b^+$, we only need to solve at most two one-dimensional optimization problems.
\begin{lemma}
\label{lem:bpbm_relation}
For $b^+$ and $b^-$ defined in \eqref{eq:bm}, \eqref{eq:bp} and \eqref{eq:bpm0}, it holds that $b^- \le b^+$.
\end{lemma}
\begin{proof}

	The desired result follows directly from the definitions of $b^+$ and $b^-$ and the monotonicity of $\partial g^*$.

\end{proof}

With the above preparations, we have the following theorem which provides an explicit formula for computing $h$.
\begin{theorem}
	\label{thm:h}
Suppose that the convex function $g:\Re \to \Re$ has bounded level-sets. For any $y\in \Re$, $x^*(y) = 
\min (b^+, \max(b^-, y))$
is an optimal solution to $\min_{x\in\Re} z_y(x)$ in \eqref{eq:h} and
\begin{equation}
\label{eq:hupdate}
h(y) =  
\begin{cases}
 g(b^-) + \lambda(b^- - y), \, &{\rm if }\; y< b^-, \\[2pt]
 g(y), \, &{\rm if }\; b^-\le y \le b^+, \\[2pt]
 g(b^+) + \mu(y - b^+), \, &{\rm if }\; y > b^+
\end{cases}
\end{equation}
with the convention that $\left\{ y\in\Re \mid y<-\infty \right\} = \left\{ y\in\Re \mid y > +\infty \right\} =  \emptyset$.
Moreover, $h:\Re \to \Re$ is  a convex function, and for any $y\in\Re$, it holds that
\begin{equation}
\label{eq:hsubg}
\partial h(y) =  
\begin{cases}
\{-\lambda\}, \, &{\rm if }\; y< b^-, \\[2pt]
[-\lambda, g'_{+}(b^-)], \, &{\rm if}\; y= b^-, \\[2pt]
\partial g(y), \, &{\rm if }\; b^-< y < b^+, \\[2pt]
[g'_{-}(b^+), \mu], \, &{\rm if}\; y= b^+, \\[2pt]
\{\mu\}, \, &{\rm if }\; y > b^+.
\end{cases}
\end{equation}
\end{theorem}
\begin{proof}
	Under the assumption that the convex function $g$ has bounded level-sets, it is not difficult to see that $z_y$ also has bounded level-sets, and thus problem \eqref{eq:h} has nonempty bounded optimal solution set.
	
	Recall the definitions of $b^-$ and $b^+$ in \eqref{eq:bm}, \eqref{eq:bp} and \eqref{eq:bpm0}. From Lemma \ref{lem:bpbm_relation}, it holds that $b^- \le b^+$. We only consider the case where both  $b^-$ and $b^+$ are finite. The case where $b^-$ and/or $b^+$ takes extended real values (i.e., $\pm\infty$) can be easily proved by slightly modifying the arguments presented here.	Now from the definitions of $b^-$ and $b^+$, we know that $\lambda, \mu$ are finite nonnegative numbers and 
	\begin{equation}
	\label{eq:fbmbp}
	-\lambda \in \partial g(b^-) \quad \mbox{and} \quad \mu \in \partial g(b^+).
	\end{equation}
	According to the value of $y$, we discuss three situations:
	\begin{itemize}
		\item[{(\rm i)}] Suppose that $y < b^-$, then $\partial p_y(b^-) = \left\{ \lambda \right\}$ and thus by \eqref{eq:subggz}, $\partial z_y(b^-) = \partial g(b^-) + \{\lambda\}$. From \eqref{eq:fbmbp}, we know that  $0\in \partial z_y(b^-)$, i.e., $b^- \in \argmin_{x\in\Re} z_y(x)$. 
		\item[{(\rm ii)}] Suppose that $y\in[b^-, b^+]$. Now, $\partial p_y(y) = [-\mu, \lambda]$ and $\partial z_y(y) = \partial g(y) + [-\mu, \lambda]$. From \cite[Theorem 24.1]{rockafellar1970convex}, it holds that 
		\[
		-\lambda =g'_-(b^-) \le g'_-(y) \le g'_+(y) \le g'_-(b^+) = \mu.
		\] 
		Then, we know from \eqref{eq:pg} that $\partial g(y)\subseteq [-\lambda, \mu]$ and thus $0\in \partial z_y(y)$, i.e.,  $y\in \argmin_{x\in\Re}  z_y(x)$.
		\item[{(\rm iii)}] Suppose that $y > b^+$, then $\partial p_y(b^+) = \left\{ -\mu \right\}$ and thus $\partial z_y(b^+) = \partial g(b^+) + \{-\mu\}$. From \eqref{eq:fbmbp}, we know that $0\in \partial z_y(b^+)$, i.e., $b^+ \in \argmin_{x\in\Re} z_y(x)$. 
	\end{itemize}
We thus proved that $x^*(y) = \min (b^+, \max(b^-, y))\in \argmin z_y(x)$. The formula of $h$ in \eqref{eq:hupdate} follows directly by observing  $h(y) = z_y(x^*(y))$ for all $y\in\Re$.

Now we turn to the convexity of $h$.
From \eqref{eq:hupdate}, we have that
\begin{equation*}
h'_+(y) =  
\begin{cases}
-\lambda, \, &{\rm if }\; y< b^-, \\[2pt]
g'_+(y), \, &{\rm if }\; b^- \le y < b^+, \\[2pt]
\mu , \, &{\rm if }\; y \ge b^+.
\end{cases}
\end{equation*}
It is not difficult to know from Lemma  \ref{lem:lfrf} and \eqref{eq:fbmbp} that $h'_+$ is non-decreasing over $\Re$. Then, the convexity of $h$ follows easily from \cite[Theorem 6.4]{hiriart2004fundamentals}.

Lastly, we note that \eqref{eq:hsubg} can be obtained from \eqref{eq:hupdate} and the fact that 
 $-\lambda = h'_-(b^-) \le h'_+(b^-) = g'_+(b^-)$ and $ g'_-(b^+) = h'_-(b^+) \le h'_+(b^+) = \mu$, i.e., $\partial h(b^-) = [-\lambda, g'_+(b^-)]$ and $\partial h(b^+) = [g'_-(b^+), \mu]$. 
We thus complete the proof for the theorem.
\end{proof}

Theorem \ref{thm:h} indicates that the optimal value function $h$ can be constructed directly from the input function $g$. Indeed, after identifying $b^-$ and $b^+$, $h$ is obtained via  replacing $g$ over $(-\infty, b^-)$ and $(b^+, +\infty)$ by simple affine functions. This procedure also results in a truncation of the subgradient of $g$. 
Indeed, from \eqref{eq:hsubg}, we know that $-\lambda \le  h_{-}^\prime(y) \le h_{+}^\prime(y) \le \mu$ for all $y\in\Re$. That is, the subgradient of $h$ is restricted between the upper bound $\mu$ and lower bound $-\lambda$. 
 To further help the understanding of Theorem \ref{thm:h}, we also provide a simple illustration here. Specifically, let $g$ be a quadratic function $g(x) = (x - 0.5)^2$, $x\in\Re$ and set $\lambda = 0.4$, $\mu = 0.2$. In this case, simple computations assert that $b^- = 0.3$ and $b^+ = 0.6$. 
 In Figure \ref{fig1:truncate}, we plot functions $g$, $h$ and their derivatives. Now, it should be more clear that the updating formula in Theorem \ref{thm:h} can be regarded as a generalization to the famous soft-thresholding \cite{Donoho1994ideal,Donoho1995denoising}.

\begin{figure}
    \centering
    \subfigure[Functions $g$ and $h$]{
        \includegraphics[width=0.43\textwidth]{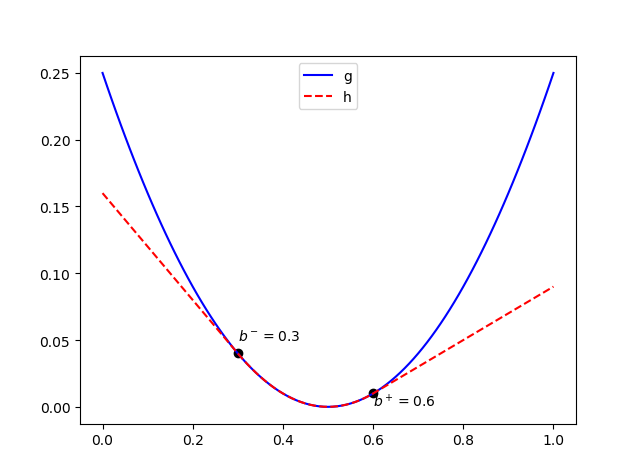}
    }
    \subfigure[Derivatives of $g$ and $h$]{
    
        \includegraphics[width=0.43\textwidth]{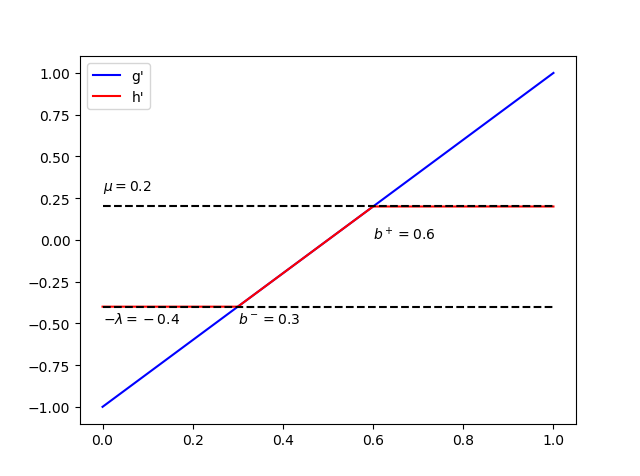}
    }

    \caption{Illustration of Theorem \ref{thm:h}, functions $g$, $h$ in the left panel and their  derivatives in the right panel.}
    \label{fig1:truncate} 
\end{figure}

Next, we turn to problem \eqref{eq:hxi}. At each iteration of Algorithm \ref{alg:DP}, to preform the ``update'' step using Theorem \ref{thm:h}, we need to verify the assumption on $g_i$ which is required in Theorem \ref{thm:h}.

\begin{proposition}\label{prop:hconvex}
{Suppose that Assumption \ref{assum:level_set} holds. Then, it holds that for all $i=1,\ldots,n$,  $g_i = f_i + h_i$ are real-valued  convex functions with bounded level-sets.}
\end{proposition}
\begin{proof}
We prove the result by induction on $i$. The result clearly holds with $i=1$ since $h_1 = 0$ and $f_1$ is assumed to be a convex function with  bounded level-sets. Now, suppose that for all $i\le l$, $g_i$ are convex and have bounded level-sets. Then, we can invoke Theorem \ref{thm:h} with $g = g_{l}$, $\lambda = \lambda_{l} $ and $\mu = \mu_{l}$ to know that $h_{l+1}:\Re \to \Re$ is convex and takes the form as in \eqref{eq:hupdate}.  It is not difficult to verify that
\[
(h_{l+1}0^+)(d) \ge 0, \quad \forall\, d\ne 0.
\]
Since $f_{l+1}$ is assumed to be a real-valued convex function with bounded level-sets, we know from Lemma \ref{lem:recessf} that 
$(f_{l+1}0^+)(d) >0$ for all $d\neq 0$ and $g_{l+1}$ is a real-valued convex function satisfying 
 \[ (g_{l+1}0^+)(d) = (f_{l+1}0^+)(d) + (h_{l+1}0^+)(d) > 0, \quad \forall\, d\neq 0.\]
 That is $g_{l+1}$ has bounded level-set and the proof is completed.

\end{proof}

If additional smoothness assumptions hold for the loss functions $f_i$, $i=1,\ldots, n$, a similar proposition on the differentiability of $h_i$ and $g_i$ can also be obtained.
\begin{proposition} \label{prop:hdiff}
Suppose that Assumption \ref{assum:level_set} holds and each $f_i$, $i=1,\ldots,n$, is differentiable. Then, both $g_i$ and $h_i$, $i=1,\ldots, n$, are differentiable functions on $\Re$.
\end{proposition}
\begin{proof}

We prove the proposition by induction on $i$.
Clearly, the assertion holds with $i=1$ since $h_1 = 0$, $g_1 = h_1 + f_1 = f_1$, and $f_1$ is assumed to be differentiable.
Now, assume that for all $i\le l$, $h_{l}$ and $g_{l}$ are differentiable. 
Then, by Proposition \ref{prop:hconvex} and  Theorem \ref{thm:h}, we know that $h_{l+1}$ is differentiable over $(-\infty, b_{l+1}^-)$, $(b_{l+1}^-,b_{l+1}^+)$, and $(b_{l+1}^+,+\infty)$. Hence, we should check the differentiability of $h$ at $b^-_{l+1}$ and $b^+_{l+1}$.
{Here, we only consider the case with $-\infty < b_{l+1}^- \le b_{l+1}^+ < +\infty$. The case where either $b_{l+1}^-$ and/or $b_{l+1}^+$ takes extended real values (i.e., $\pm\infty$) can be easily proved by slightly modifying the arguments presented here.}
Recalling definitions of $b_{l+1}^-$ and $b_{l+1}^+$ in \eqref{eq:bm}  and \eqref{eq:bp}, and using the differentiability of $g_{l}$, we have $g'_{l+}(b_{l+1}^-)=g_l'(b_{l+1}^-) = -\lambda_l$ and $g'_{l-}(b_{l+1}^+) =g_l'(b_{l+1}^+) = \mu_l$. Then, \eqref{eq:hsubg}  in Theorem \ref{thm:h} implies that
\[
\partial h_{l+1}(b_{l+1}^-) = [-\lambda_l, g'_{l+}(b_{l+1}^-)] = \{ - \lambda_l \} \mbox{ and } \partial h_{l+1}(b_{l+1}^+) = [g'_{-}(b_{l+1}^+), \mu_l] =  \{ \mu_l \}.
\]
 Hence, $h_{l+1}$ and $g_{l+1} = h_{l+1} + f_{l+1}$ are differentiable over $\Re$. We thus complete the proof.
\end{proof}

With the above discussions, in particular, Theorem \ref{thm:h} and Proposition \ref{prop:hconvex}, we can write the ``update'' step in Algorithm \ref{alg:DP} in a more detailed fashion.
\begin{algorithm}[H]
\caption{The ``update'' step in Algorithm \ref{alg:DP}: $[h,b^-,b^+] = \mbox{\bf update} (g, \lambda , \mu)$}
\label{alg:update}
\begin{algorithmic}[1]
\State {{\bf Input} function $g$ and parameters $\lambda,\mu$. }
\State {Compute $(b^-,b^+)$ according to definitions \eqref{eq:bm},  \eqref{eq:bp} and \eqref{eq:fbmbp}.}
\State {Compute $h$ via \eqref{eq:hupdate}: for all $y\in\Re$,}
\State {\begin{equation*}
	h(y) =  
	\begin{cases}
	g(b^-) + \lambda(b^- - y), \, &{\rm if }\; y< b^-, \\[2pt]
	g(y), \, &{\rm if }\; b^-\le y \le b^+, \\[2pt]
	g(b^+) + \mu(y - b^+), \, &{\rm if }\; y > b^+.
	\end{cases} 
	\end{equation*}}
\State {\textbf{Return} $h, b^-, b^+$ }
\end{algorithmic}
\end{algorithm}

\noindent Meanwhile, we can further obtain the implementation details for the ``recover'' step in Algorithm \ref{alg:DP} based on the discussions in \eqref{eq:backward_xi} and Theorem \ref{thm:h}.
\begin{algorithm}[H]
\caption{ The ``recover'' step in Algorithm \ref{alg:DP}: $[x_1,\ldots, x_{n-1}] = \mbox{\bf recover}(x_n,\{ (b_i^-,b_i^+) \}_{i=2}^{n})$}
\label{alg:recover}
\begin{algorithmic}[1]
\State{{\bf Input} $x_n$ and $\{ (b_i^-,b_i^+) \}_{i=2}^{n})$} 
\For {$i = n: 2$}
\State {$x_{i-1} = \min (b_i^+, \max(b_i^-, x_i))$} 
\EndFor
\State {\textbf{Return} $x_1,\cdots,x_{n-1}$ }
\end{algorithmic}
\end{algorithm}

Thus, instead of using the definition  \eqref{eq:h} directly to compute $h_i$, we leverage on the special structure of the nearly isotonic regularizer and show that $h_i$ can be explicitly constructed by solving at most two one-dimensional optimization problems. Hence, we obtain an implementable dynamic programming algorithm for solving the generalized nearly isotonic optimization problem \eqref{eq:gip}. 
\begin{remark}
\label{rmk:bmbp}
One issue we do not touch seriously here is the computational details of obtaining $b^-$ and $b^+$ in Algorithm 2. Based on the definitions in \eqref{eq:bm}, \eqref{eq:bp} and \eqref{eq:bpm0}, to obtain $b^-$ and $b^+$, at most two one-dimensional optimization problems in the form of \eqref{eq:omega} need to be solved.
Given the available information of $g$ (such as the function value and/or subgradient of $g$ at given points),  various one-dimensional optimization algorithms can be used. Moreover, in many real applications such as the later discussed $\ell_1$-GNIO and $\ell_2$-GNIO problems, $b^-$ and $b^+$ can be computed via closed-form expressions.  
\end{remark}

\section{Implementation details of Algorithm \ref{alg:DP} for $\ell_2$-GNIO}
\label{sec:imple_quadratic}
In the previous section, an implementable DP {based} algorithmic framework is developed for solving the GNIO problem \eqref{eq:gip}. We shall mention that for special classes of loss functions, it is possible to obtain a low running complexity implementation of Algorithm \ref{alg:DP}. As a prominent example, in this section, we discuss some implementation details of Algorithm \ref{alg:DP} for solving $\ell_2$-GNIO problems, i.e., for all $i=1,\ldots n$, each loss function $f_i$ in \eqref{eq:gip} is a simple quadratic function. Specifically, given data points $\{y_i\}_{i=1}^n$ and positive weights $\{w_i\}_{i=1}^n$, we consider the convex quadratic loss functions
\begin{equation}
\label{eq:l2loss}
f_i(x_i) = w_i (x_i - y_i)^2, \quad  x_i \in\Re, \quad i=1,\ldots, n.
\end{equation}  and the following 
problem:
\begin{equation*}
\mbox{($\ell_2$-GNIO)} \qquad 
\min_{x \in \Re^n} \quad \sum_{i=1}^{n} w_i(x_i - y_i)^2 + \sum_{i=1}^{n-1} \lambda_{i} (x_i - x_{i+1})_{+} + \sum_{i=1}^{n-1} \mu_{i} (x_{i+1} - x_{i})_{+},
\end{equation*}
We note that quadratic loss functions have been extensively used in the context of shape restricted statistical regression problems \cite{Bartholomew1972statistic,Hoefling2009path,Ryu2004medical,Tibshirani2016nearly}.

In the following, special numerical representations of quadratic functions will be introduced to achieve a highly 
efficient implementation of Algorithm \ref{alg:DP} for solving $\ell_2$-GNIO problems. We start with a short introduction of univariate piecewise quadratic functions. We say $h$ is a univariate piecewise quadratic function if there is a strictly increasing sequence $\{ \beta_l \}_{l=1}^{K} \subseteq \Re$ and $h$ agrees with a quadratic function on each of the intervals $(-\infty,\beta_1)$, $[\beta_l,\beta_{l+1})$, $l=1,\ldots, K-1$, and $[\beta_K,+\infty)$. Here, each $\beta_l$ is referred to as a ``breakpoint'' of $h$ and univariate affine functions are regarded as degenerate quadratic functions. The following proposition states that for $\ell_2$-GNIO problems, in each iteration of Algorithm \ref{alg:DP}, the corresponding functions $h_i$ and $g_i$, $i=1,\ldots,n$, are all convex differentiable piecewise quadratic functions. The proof of the proposition is similar to that of Propositions \ref{prop:hconvex} and \ref{prop:hdiff} and is thus omitted. 
\begin{proposition}
	\label{prop:quad}
	Let the loss functions $f_i$, $i=1,\ldots, n$, in GNIO problem \eqref{eq:gip} be convex quadratic functions as given in \eqref{eq:l2loss}. Then, in Algorithm \ref{alg:DP}, all involved functions $h_i$ and $g_i$, $i=1,\ldots, n$, are univariate convex differentiable piecewise quadratic functions. Moreover, it holds that
	\[
	\inf_x g_i'(x) = -\infty\quad  \mbox{ and } \quad 
	\sup_x g_i'(x) = + \infty, \quad \forall\, i=1,\ldots,n.
	\]
\end{proposition}

From Proposition \ref{prop:quad}, we see that an important issue in the implementation of Algorithm \ref{alg:DP} is the numerical representation of the univariate convex differentiable piecewise quadratic functions. Here, inspired by the data structures exploited in \cite{Rote2019isoDP,Johnson2013dplasso}, we adopt a strategy called ``difference of coefficients'' to represent these functions. Based on these representations, implementation details of the subroutines ``sum'' and ``update'' in Algorithm \ref{alg:DP} will be further discussed.

Let $h$ be a univariate convex differentiable piecewise quadratic function and $\{\beta_l\}_{l=1}^{K}$ be the associated breakpoints.
These breakpoints define $K+1$ intervals in the following form:
\[
(-\infty,\beta_1), \, [\beta_l,\beta_{l+1}), \, l = 1,\ldots, K-1,  \mbox{ and } [\beta_K,+\infty).
\]
Assume that $h$ agrees with $a_j x^2 +  b_j x$ on the $j$-th ($1 \le j \le K+1$) interval\footnote{The intercepts are ignored for all the pieces as they are irrelevant in the optimization process.}. Here, $a_j \ge 0$, $b_j \in \Re$ are given data. To represent $h$, we first store the breakpoints $\{\beta_l\}_{l=1}^{K}$ in a sorted list $B^h$ with ascending order and store the number of breakpoints as $K^h = K$. 
Associated with each breakpoint $\beta_l$, we compute the difference of the coefficients between the consecutive piece and the current piece:
 $d_l = (a_{l+1} -a_l , b_{l+1} - b_l )$, $l=1,\ldots, K$ and store $\{d_l\}_{l=1}^{K}$ in list $D^h$. Information of the leftmost and rightmost pieces are stored in two tuples $c_L^h = (a_1,b_1)$ and $c^h_R = (a_{K+1},b_{K+1})$, respectively. With these notation, we can write the representation symbolically as 
 \[
 h = [ B^h, D^h, c_L^h, c_R^h, K^h].
 \]
 We summarize the above representation of $h$ in Table \ref{table:data structure}. 
\begin{table}
	\centering
	\begin{tabular}{|l|l|l|}
		\hline
		name     & data-type  & \multicolumn{1}{c|}{explaination} \\ \hline
		$B^h$     & list  & breaking points                 \\ \hline
		$D^h$   & list   & differences of coefficients       \\ \hline
		$c^h_L$ & tuple & coefficients of the leftmost piece   \\ \hline
		$c^h_R$ & tuple & coefficients of the rightmost piece \\ \hline
		$K^h$     & integer  & number of breaking points                 \\ \hline
	\end{tabular}
	\caption{Representation of the univariate convex differentiable piecewise quadratic function $h$.}
	\label{table:data structure} 
\end{table}
\noindent In the same spirit, given $y\in\Re$, for the	 quadratic function $f(x) = w (x-y)^2$, $x\in\Re$, we have the following representation $f = [  \emptyset,  \emptyset, (w, -2wy), (w, -2wy), 0]$.
With the help of these representations, for $\ell_2$-GNIO problems, one can easily perform the ``sum'' step in Algorithm \ref{alg:DP}. 
More specifically, the representation of $g=h+f$ can be written in the following way
\[
g = [B^g, D^g, c_L^g, c_R^g, K^g] = [B^h, D^h, c_L^h + (w, -2wy), c_R^h + (w, -2wy), K^h].
\]
That is, to obtain the representation of $g$, one merely needs to modify the coefficients of the leftmost and rightmost pieces, i.e., $c_L^h$ and $c_R^h$, in the representation of $h$. Hence, with this representation, the running time for the above ``sum'' step is ${\cal O}(1)$.
We summarize the above discussions into the following framework. 
\begin{algorithm}
	\caption{``sum'' in Algorithm \ref{alg:DP} for $\ell_2$-GNIO: $g =\mbox{{\bf sum}}\, (h,f)$}
	\label{alg:sum}
	\begin{algorithmic}[1]
		\State {{\bf Input} representations $h = [B, D, c_L, c_R, K]$ and $f = [ \emptyset,  \emptyset, (w, -2wy), (w, -2wy), 0]$}
		\State { Compute the representation of $g$ via 
	\begin{equation*}
	[B, D, c_L, c_R, K] = [B, D, c_L + (w, -2wy), c_R + (w, -2wy), K].
\end{equation*}
}\vspace{-0.4cm}
		\State {\textbf{Return} $g = [B, D, c_L, c_R, K]$. }
	\end{algorithmic}
\end{algorithm}

In the following, we show how to use the above representations to efficiently implement the ``update'' step in Algorithm \ref{alg:DP} for solving $\ell_2$-GNIO problem. By Proposition \ref{prop:quad}, we know that in each iteration of Algorithm \ref{alg:DP}, the output $b_i^-$ (respectively, $b_i^+$) of the ``update'' step will always be finite real numbers except the case with $\lambda_{i-1}$ (respectively, $\mu_{i-1}$) being $+\infty$. Hence, we focus on the computations of finite $b_i^-$ and/or $b_i^+$. Consider an abstract instance where a univariate convex differentiable piecewise quadratic function $g = [B^g, D^g, c^g_L, c^g_R, K^g]$ and positive parameters $\lambda$ and $\mu$ are given. We aim to compute $[h, b^-, b^+] = {\rm update}(g, \lambda, \mu)$ and the corresponding representation of $h = [B^h, D^h, c^h_L, c^h_R, K^h]$. Specifically, from definition \eqref{eq:bm}, we compute $b^-$ by solving the following equation: 
\[
g'(b^-) + \lambda = 0.
\] 
Based on the representation of $g$, the derivative of $g$ over all pieces can be easily obtained. Hence, $b^-$ can be efficiently determined by searching from the leftmost piece to the rightmost piece of $g$. Similarly, $b^+$ can be computed in a reverse searching order. The implementation details are presented in Algorithm \ref{alg:truncate}. In fact, as one can observe, the computations of $b^-$ and $b^+$ can be done simultaneously. For the simplicity, we do not exploit this parallelable feature in our implementation. We also have the following observations in Algorithm \ref{alg:truncate}: %
(1) in each round of Algorithm \ref{alg:truncate}, the number of breakpoints, denoted $\Delta$, satisfies $\Delta = k^- + k^+ -2 \le K^g$ and 
$K^h \le K^g - \Delta + 2$; (2) since $b^- \le b^+$, each breakpoint can be marked as ``to delete'' at most once; (3) the total number of the executions of the while-loops is the same as the number of the deleted breakpoints; (4) the output $B^h$ is  a sorted list.

\begin{algorithm}
	\caption{ ``update'' in Algorithm \ref{alg:DP} for $\ell_2$-GNIO: $[h,b^-,b^+] = \mbox{{\bf update}}\, (g,\lambda,\mu)$ }
	\label{alg:truncate}
	\begin{algorithmic}[1]
		\State {{\bf Input} the representation $g = [B, D, c_L, c_R, K]$ and parameters $\lambda$, $\mu$}
		\\
		\State { $k^- = k^+ = 1$}
		
		\If {$\lambda = + \infty$}
		\State { $b^- = -\infty$ }
        \Else
		\State {$(a,b) = c_L$}
		\If {$K \ge 1$}
		\While {$2a B\{k^-\} + b < -\lambda$ }  \Comment{$B\{k^-\}$ denotes the $k^-$th element in $B$}
		\State {$(a,b) =  (a,b) + D\{k^-\}$} \Comment{$D\{k^-\}$ denotes the $k^-$th element in $D$}
		\State {Mark the $k^-$th breakpoint as ``to delete'', $k^- = k^-+1$}
		\EndWhile
		\EndIf
		\State {$b^- = -  \frac{b +\lambda}{2a}, \, c_L = (0, -\lambda)$ } \Comment{Found the ``correct'' piece and $g'(b^-) = 2ab^- + b = -\lambda$}
		\EndIf
		\\
		\If {$\mu = +\infty$}
		\State { $b^+ = +\infty$ }
		\ElsIf {$\lambda = \mu = 0$}
		\State{$b^+ = b^-$, mark the $k^-, \ldots, K$th breakpoints as ``to delete'', $k^+ = K - k^-+2$}
		\Else
		\State {$(a',b') = c_R$  }
		\If {$K\ge 1$}
		\While {$2a' B\{K - k^+ + 1\} + b' > \mu$}
		\State {$(a',b') =(a',b') - D\{K - k^+ + 1\}$} 
        \State {Mark the $(K - k^+ + 1)$th breakpoint as ``to delete'', $k^+ = k^+ + 1$}
		\EndWhile
		\EndIf
		\State {$b^+ = \frac{ \mu - b'}{2a'},\, c_R = (0,\mu)$ } \Comment{Found the ``correct'' piece and $g'(b^+) = 2a' b^+ + b' = \mu$}
		\EndIf

\\
		\State {Delete $(k^- + k^+ -2)$ breakpoints marked as ``to delete'' from $B$ and the corresponding coefficients from $D$} \Comment{Since $b^- \le b^+$, each breakpoint can be marked as ``to delete'' at most once}
\\
		\If {$\lambda < + \infty$}
		\State{$B = B.prepend(b^-), \, D = D.prepend((a,b) - c_L), \, K = K - (k^- + k^+ -2) + 1$}
		\EndIf
		\If {$\mu < +\infty$ and $\lambda + \mu > 0$}
		\State {$B = B.append(b^+), \, D = D.append(c_R - (a',b')), \, K = K + 1$}
		\EndIf
		\State {\textbf{Return} $h = [B, D, c_L, c_R, K]$ and $b^-,b^+$ }
	\end{algorithmic}
\end{algorithm}

Now, we are ready to discuss the worst-case running time of Algorithm \ref{alg:DP} for solving $\ell_2$-GNIO problems. We first  observe that the ``sum'' and ``recover'' steps in Algorithm \ref{alg:DP} takes ${\cal O}(n)$ time in total.  
For $i=1, \ldots, n-1$, we denote the number of breakpoints of $g_i$ by $K_i$ and the  number of deleted breakpoints from $g_i$ by $\Delta_i$. From the observation (1), we know that $\Delta_i \le K_i$ and $K_{i+1} \le K_i - \Delta_i + 2$ for $i=1,\ldots, n-1$.  Simple calculations yield 
$
	\sum_{i=1}^{n-1}\Delta_i \le 2(n - 2)
$, i.e., the total number of  the deleted breakpoints in Algorithm \ref{alg:DP} is upper bounded by $2n$.
This, together with the observation (3), implies that the total  number of executions of the while-loop in Algorithm \ref{alg:truncate} is upper bounded by $2n$. Thus, the overall running time of the ``update'' step is ${\cal O}(n)$. Therefore, based on the ``difference-of-coefficients'' representation scheme and the special implementations of the ``sum'' and the ``update'' steps,  we show that Algorithm \ref{alg:DP} solves  $\ell_2$-GNIO with ${\cal O}(n)$ worst-case running time. That is our specially implemented Algorithm \ref{alg:DP} is an optimal algorithm for solving  $\ell_2$-GNIO problems.

\begin{remark}
	\label{rmk:l1rmk}
{
	In the above discussions, in addition to the numerical representations of convex differentiable piecewise quadratic functions, we see that the data structures also play important roles in both the implementations and the running time analysis. Hence,  in order to obtain efficient implementations of Algorithm \ref{alg:DP} for solving other GNIO problems, one has to explore appropriate problem dependent data structures. For example, consider the $\ell_1$-GNIO problem, i.e., given $y\in\Re^n$ and a positive weight vector $w\in\Re^n$, let the loss functions in \eqref{eq:gip} take the following form: $f_i(x_i) = w_i |x_i - y_i|$, $x_i\in\Re$, $i=1,\ldots,n$. Due to the nonsmoothness of $f_i$, if the data structure ``list'' were used to store the involved breakpoints, the running time of each ``sum'' step in Algorithm \ref{alg:DP} will be ${\cal O}(n)$ and the overall running time will be ${\cal O}(n^2)$ in the worst case. For the acceleration, we propose to use red-black trees \cite{Cormen2009intro} to store these breakpoints. It can be shown that with this special data structure, the running time of each ``sum'' step can be reduced to ${\cal O}(\log n)$. Hence, the overall running time of the ``sum'' step, as well as Algorithm \ref{alg:DP}, is ${\cal O}(n \log n)$. 
	Since the running time analysis is not the main focus of the current paper, more details on the ${\cal O}(n \log n)$ complexity of Algorithm \ref{alg:DP} for solving $\ell_1$-GNIO are presented in the Appendix. This result is also verified by various numerical evidences in Section \ref{sec:num}.
	Finally, we shall mention that the ${\cal O}(n\log n)$ running time matches the best-known results \cite{stout2019fastest} for some special $\ell_1$-GNIO problems, such as $\ell_1$-isotonic regression and  $\ell_1$-unimodal regression problems.
}
\end{remark}

\section{Numerical experiments}
\label{sec:num}
In this section, we shall evaluate the performance of our dynamic programming Algorithm \ref{alg:DP} for solving generalized nearly isotonic regression problems \eqref{eq:gip}. Since the $\ell_1$ losses and $\ell_2$ losses are widely used in practical applications, we implement Algorithm \ref{alg:DP} for solving both $\ell_1$-GNIO and $\ell_2$-GNIO problems:
\begin{equation}
\label{eq:abs}
\min_{x\in \Re^n} \quad \sum_{i=1}^{n} w_i|x_i - y_i| + \sum_{i=1}^{n-1} \lambda_{i} (x_i - x_{i+1})_{+} + \sum_{i=1}^{n-1} \mu_{i} (x_{i+1} - x_{i})_{+},
\end{equation} 
and
\begin{equation}
\label{eq:quad}
\min_{x \in \Re^n} \quad \sum_{i=1}^{n} w_i(x_i - y_i)^2 + \sum_{i=1}^{n-1} \lambda_{i} (x_i - x_{i+1})_{+} + \sum_{i=1}^{n-1} \mu_{i} (x_{i+1} - x_{i})_{+},
\end{equation}
where $y\in\Re^n$ is a given vector, $\{w_i\}_{i=1}^n$ are positive weights, $\lambda_i$ and $\mu_i$, $i=1,\ldots,n-1$ are nonnegative and possibly infinite parameters.
We test our algorithms using both randomly generated data and the real data from various sources.
For the testing purposes, we vary the choices of parameters $\lambda_i$ and $\mu_i$, $i=1,\ldots,n-1$.
Our algorithms are implemented in {C/C++}\footnote{The code is available at \url{https://github.com/chenxuyu-opt/DP_for_GNIO_CODE} (DOI:10.5281/zenodo.7172254).} and all the computational results are obtained on a laptop (Intel(R) Core i7-10875H CPU at 2.30 GHz, 32G RAM, and 64-bit Windows 10 operating system).

Note that our algorithm, when applied to solve problems \eqref{eq:abs} and \eqref{eq:quad}, is an exact algorithm if the rounding errors are ignored. As far as we know, there is currently no other {open-access} exact algorithm which can simultaneously solve both problems \eqref{eq:abs} and \eqref{eq:quad} under various settings of parameters.
Fortunately, these two problems can be equivalently rewritten as linear programming and convex quadratic programming problems, respectively. Hence, in first two subsections of our experiments, we compare our algorithm with Gurobi (academic license, version 9.0.3), which can robustly produce high accurate solutions and is among the most powerful commercial solvers for linear and quadratic programming. For the Gurobi solver, we use the default
parameter settings, i.e., using the default stopping tolerance and all computing cores. 
For all the experiments,
our algorithm and Gurobi output objective values whose relative gaps are of order $10^{-8}$, i.e., both of them produce highly accurate solutions.

To further examine the efficiency of our dynamic programming algorithm, in the last part of our numerical section, we conduct more experiments on solving the fused lasso problem \eqref{prob:flsa}. Specifically, we compare our algorithm with the C implementation\footnote{ { \url{ https://lcondat.github.io/download/Condat_TV_1D_v2.c }}} of Condat's direct algorithm \cite{Condat2013direct}, which, based on the extensive evaluations in \cite{Barbero2018modular}, appears to be one of the most efficient algorithm specially designed and implemented  for solving the large-scale fused lasso problem \eqref{prob:flsa}.

\subsection{DP algorithm versus Gurobi: Simulated data}
We first test our algorithm with simulated data sets. Specifically, the input vector $y\in\Re^n$ is set to be a random vector with i.i.d. uniform ${\cal U}(-100,100)$ entries. The problems sizes, i.e., $n$, vary from $10^4$ to {$10^7$}.
The positive weights $\{w_i\}_{i=1}^n$ are generated in three ways: (1) fixed, i.e., $w_i = 1$ in $\ell_1$-GNIO \eqref{eq:abs} and $w_i = 1/2$ in $\ell_2$-GNIO \eqref{eq:quad}, $i=1,\ldots,n$; (2) i.i.d. sampled from the uniform distribution ${\cal U}(10^{-2}, 10^2)$; (3) i.i.d. sampled from Gaussian distribution ${\cal N}(100, 100)$ with possible nonpositive outcomes replaced by $1$.  
For all  problems, we test {seven} settings of parameters $\{\lambda_i\}_{i=1}^{n-1}$ and $\{\mu_i\}_{i=1}^{n-1}$:
\begin{enumerate}
	\item Isotonic: $\lambda_i = +\infty$, $\mu_i = 0$ for $i=1,\ldots, n-1$;
	\item Nearly-Isotonic: $\lambda_i = \log(n)$, $\mu_i = 0$ for $i=1,\ldots, n-1$;

	\item Unimodal: $\lambda_i = +\infty, \mu_i = 0$ for $i=1,\ldots, m$ and $\lambda_i = 0, \mu_i = +\infty$ for $i = m+1, \ldots, n-1$ with $m = [\frac{n-1}{2}]$;
	
		\item Fused: $\lambda_i = \mu_i = \log(n)$ for $i=1,\ldots, n-1$;
	\item Uniform: All $\lambda_i$ and $\mu_i$, $i=1,\ldots,n-1$ are i.i.d. sampled from the uniform distribution ${\cal U}(0, 10^3)$;
	\item Gaussian: All $\lambda_i$ and $\mu_i$, $i=1,\ldots,n-1$ are i.i.d. sampled from Gaussian distribution ${\cal N}(100,100)$ with possible negative outcomes set to be $0$;
	\item {Mixed:  $\lambda_i = +\infty$ for $i = 1,\ldots, [\frac{n}{5}]$ and $\mu_i = +\infty$ for $i = n-[\frac{n}{5}],\ldots, n-1$; 
$\lambda_i$, $i = [\frac{n}{5}] + 1,\ldots,n-1$ and $\mu_i$, $i = 1, \ldots, n-1-[\frac{n}{5}]$ are i.i.d. sampled from the uniform distribution ${\cal U}(0, 10^3)$.}
\end{enumerate}
Table \ref{table:l1w} reports the detailed numerical results of Algorithm \ref{alg:DP} for solving $\ell_1$-GNIO problem \eqref{eq:abs} under the above mentioned different settings. As one can observe, our algorithm is quite robust to various patterns of weights and parameters. Moreover, the computation time scales near linearly with problem dimension $n$, which empirically verifies our theoretical result on the worst-case ${\cal O}(n\log n)$ running time of Algorithm \ref{alg:DP} for $\ell_1$-GNIO problems in Remark \ref{rmk:l1rmk}.

\begin{table}
\begin{small}
	\centering
	\begin{tabular}{@{}cccccccc@{}}
		\toprule
		\multicolumn{8}{c}{ { Runtime  of Algorithm \ref{alg:DP} for \eqref{eq:abs} with $w_i = 1$ } }                                 \\ \midrule
		$n$   & Isotonic & Nearly-isotonic & Unimodal & Fused  & Uniform & Gaussian & Mixed \\ \midrule
		1e4 & 0.003      & 0.003           &  0.003       & 0.003    & 0.002    & 0.003   & 0.002 \\
		1e5 & 0.046      & 0.036           &  0.042       & 0.031    & 0.019    & 0.022   & 0.019\\
		1e6 & 0.582      & 0.306           &  0.545       & 0.298    & 0.189    & 0.214   & 0.228\\ 
		1e7 & 7.685      & 3.040           &  7.678       & 3.107    & 1.839    & 2.116   & 2.938 \\ 
		 \bottomrule
				\toprule
		\multicolumn{8}{c}{  {Runtime of Algorithm \ref{alg:DP} for \eqref{eq:abs} with $w_i \sim  {\cal U}(10^{-2}, 10^2)$ } }                                 \\ \midrule
		$n$   & Isotonic & Nearly-isotonic & Unimodal & Fused  & Uniform & Gaussian  & Mixed \\ \midrule
		1e4 & 0.004      & 0.003           &  0.003       & 0.002    &  0.002     & 0.001   & 0.002 \\
		1e5 & 0.048      & 0.025           &  0.041       & 0.027    &  0.023     & 0.012   & 0.209\\
		1e6 & 0.701      & 0.250           &  0.558       & 0.255    &  0.195     & 0.261   & 0.239 \\
		1e7 & 9.295      & 2.461           &  8.052       & 2.263    &  1.877     & 2.665   & 3.102 \\ 
		\bottomrule
			\toprule
		\multicolumn{8}{c}{ { Runtime of Algorithm \ref{alg:DP} for \eqref{eq:abs} with $w_i \sim  {\cal N}(100, 100)$ } }                                 \\ \midrule
		$n$   & Isotonic & Nearly-isotonic & Unimodal & Fused  & Uniform & Gaussian & Mixed\\ \midrule
		1e4 & 0.004      & 0.002           &  0.004       & 0.002    &  0.002     & 0.003  & 0.002 \\
		1e5 & 0.053      & 0.024           &  0.043       & 0.025    &  0.025     & 0.028  & 0.209 \\
		1e6 & 0.718      & 0.244           &  0.574       & 0.249    &  0.211     & 0.271  & 0.244 \\
		1e7 & 9.416      & 2.492           &  8.531       & 2.258    &  1.927     & 2.565  & 3.204 \\ 
		\bottomrule
	\end{tabular}
	\caption{Runtime (in seconds) of Algorithm \ref{alg:DP} for solving $\ell_1$-GNIO problem \eqref{eq:abs} under different settings. Results are averaged over 10 simulations.}
	\label{table:l1w}
\end{small}
\end{table}

Note that when parameters $\lambda_i$ and $\mu_i$ are finite, $\ell_1$-GNIO problem \eqref{eq:abs} can be formulated as the following linear programming problem (see, e.g., \cite[Section 7]{Hochbaum2017faster}):
\begin{equation}
\label{prob:LPgip}
\begin{aligned}
\min_{x,z\in \Re^n, u,v \in \Re^{n-1} } \quad  & \sum_{i=1}^{n} w_i z_i +  \sum_{i=1}^{n-1} \lambda_i u_i + \sum_{i=1}^{n-1} \mu_i v_i  \\
\text{s.t.} \quad &   z_i \ge y_i - x_i , \quad i = 1, \cdots,n, \\
&    z_i \ge x_i - y_i , \quad i = 1, \cdots,n, \\
&    x_i - x_{i+1} \le u_i  , \quad i = 1, \cdots,n-1,\\
&    x_{i+1} - x_i \le v_i  , \quad i = 1, \cdots,n-1,\\
&    u \ge 0 ,  \, v \ge 0.
\end{aligned} 
\end{equation}
If some $\lambda_i$ and/or $\mu_i$ is infinite, certain modifications need to be considered. For example,  the isotonic regression \eqref{prob:iso_reg} can be equivalently reformulated as:
\begin{equation}
\label{prob:LPiso}
\begin{aligned}
\min_{x,z\in \Re^n} \quad  & \sum_{i=1}^{n} w_i z_i  \\
\text{s.t.} \quad &   z_i \ge y_i - x_i , \quad i = 1, \cdots,n, \\
&    z_i \ge x_i - y_i , \quad i = 1, \cdots,n, \\
&    x_i - x_{i+1} \le 0  , \quad i = 1, \cdots,n-1.
\end{aligned} 
\end{equation}
Now, we can compare our dynamic programming algorithm with Gurobi on randomly generated data sets under various settings of parameters.
For simplicity, we only consider the fixed weights here, i.e., $w_i \equiv 1$, $i=1,\ldots,n$. 
As one can observe in Table \ref{table:gurobi}, for all the test instances, our dynamic programing algorithm outperforms Gurobi by a significant margin. {Specifically, for {17} out of {21} instances, our algorithm can be at least {220} times faster than Gurobi. Moreover, for one instance in the Gaussian setting, our algorithm can be up to {5,662} times faster than Gurobi}.  
\begin{table}
\begin{small}
	\centering
	\begin{tabular}{@{}ccccc@{}}
		\toprule
		{$n$}   & parameters pattern & $t_{\rm DP}$ & $t_{\rm Gurobi}$   & \multicolumn{1}{c}{$t_{\rm Gurobi}/t_{\rm DP}$}     \\ \midrule
		1e4 & Isotonic             & 0.003            & 0.46         & 153.33    \\	
		1e5 & Isotonic             & 0.03             & 9.79         & 326.33    \\
		1e6 & Isotonic    		   & 0.58             & 358.82       & 620.38    \\
		1e7 & Isotonic    		   & 7.69             & *           & *    \\
		1e4 & Nearly-isotonic      & 0.003            & 0.34         & 113.33    \\
		1e5 & Nearly-isotonic      & 0.04             & 5.94         & 148.50    \\
		1e6 & Nearly-isotonic      & 0.31             & 154.15  	    & 497.26    \\
		1e7 & Nearly-isotonic      & 3.04             & *  	        & *    \\
		1e4 & Unimodal             & 0.003            & 0.43		    & 143.33    \\
		1e5 & Unimodal             & 0.04             & 9.55		    & 238.75    \\
		1e6 & Unimodal             & 0.55             & 363.32       & 660.58    \\
		1e7 & Unimodal             & 7.69             & *            & *    \\
		1e4 & Fused                & 0.003            & 0.68         & 226.67    \\
		1e5 & Fused                & 0.03             & 14.21        & 473.67    \\
		1e6 & Fused                & 0.30             & 299.21       & 997.36         \\
		1e7 & Fused                & 3.11             & *            & *        \\
		1e4 & Uniform              & 0.002            & 0.65     	 & 325.00         \\
		1e5 & Uniform              & 0.02             & 15.11        & 755.50         \\
		1e6 & Uniform              & 0.19             & 197.12       & 1037.47     \\
		1e7 & Uniform              & 1.84             & *       & *     \\
		1e4 & Gaussian             & 0.003            & 2.02 	     & 673.33        \\
		1e5 & Gaussian             & 0.02             & 58.14        & 2907.00        \\
		1e6 & Gaussian             & 0.21             & 1180.75      & 5662.62       \\
		1e7 & Gaussian             & 2.12             & *            & *       \\			
		1e4 & Mixed                & 0.002            & 0.79         & 395.00    \\	
		1e5 & Mixed                & 0.02             & 14.37        & 718.50    \\
		1e6 & Mixed    		       & 0.23             & 535.19       & 2326.91   \\
		1e7 & Mixed    		       & 2.94             & *            & *    \\
		 \bottomrule
	\end{tabular}
	\caption{Runtime (in seconds) comparisons between our dynamic programming Algorithm \ref{alg:DP} and Gurobi for solving problem \eqref{eq:abs} under different settings. Results are averaged over 10 simulations. {The entry ``*" indicates that Gurobi reports ``out of memory''}.}
	\label{table:gurobi}
\end{small}
\end{table}

Next, we test our dynamic programming algorithm for solving $\ell_2$-GNIO problem \eqref{eq:quad}. Similarly, experiments are done under various settings of weights and parameters. From Table \ref{table:l2w}, we see that our algorithm can robustly solve various $\ell_2$-GNIO problems and the computation time scales linearly with the problem dimension $n$. This matches our result on the ${\cal O}(n)$ running time of Algorithm \ref{alg:DP} for solving $\ell_2$-GNIO problems.

\begin{table}
\begin{small}
	\centering
	\begin{tabular}{@{}cccccccc@{}}
		\toprule
		\multicolumn{8}{c}{ { Runtime }   of Algorithm \ref{alg:DP} for \eqref{eq:quad} with $w_i = 1/2$  }                                   \\ \midrule
		$n$   & Isotonic & Nearly-isotonic & Unimodal & Fused  & Uniform & Gaussian  & Mixed \\ \midrule
		1e4 & 0.000     & 0.000            & 0.000      & 0.000  & 0.000    & 0.000  & 0.000  \\
		1e5 & 0.002     & 0.003            & 0.002      & 0.003  & 0.003    & 0.003  & 0.002   \\
		1e6 & 0.020     & 0.031            & 0.019      & 0.032  & 0.033    & 0.031  & 0.020    \\
		1e7 & 0.206     & 0.311            & 0.197      & 0.319  & 0.334    & 0.309  & 0.198    \\
	    \bottomrule
		\toprule
		\multicolumn{8}{c}{{ Runtime }    of Algorithm \ref{alg:DP} for \eqref{eq:quad}  with $w_i \sim  {\cal U}(10^{-2}, 10^2)$}                                   \\ \midrule
		$n$   & Isotonic & Nearly-isotonic & Unimodal & Fused  & Uniform & Gaussian  & Mixed  \\ \midrule
		1e4 & 0.000     & 0.000            & 0.000      & 0.000  & 0.000    & 0.000   & 0.000  \\
		1e5 & 0.002     & 0.003            & 0.002      & 0.003  & 0.003    & 0.003    & 0.002 \\
		1e6 & 0.021     & 0.029            & 0.022      & 0.030  & 0.031    & 0.029   & 0.021  \\
		1e7 & 0.213     & 0.304            & 0.224      & 0.303  & 0.315    & 0.296   & 0.185  \\
		\bottomrule
		\toprule
		\multicolumn{8}{c}{{ Runtime }   of Algorithm \ref{alg:DP} for \eqref{eq:quad}  with $w_i \sim  {\cal N}(100, 100)$}                                   \\ \midrule
		$n$   & Isotonic & Nearly-isotonic & Unimodal & Fused  & Uniform & Gaussian & Mixed   \\ \midrule
		1e4 & 0.000     & 0.000            & 0.000      & 0.000  & 0.000    & 0.000  & 0.000 \\
		1e5 & 0.002     & 0.003            & 0.002      & 0.003  & 0.003    & 0.003  & 0.002 \\
		1e6 & 0.020     & 0.029            & 0.020      & 0.029  & 0.030    & 0.029  & 0.018 \\
        1e7 & 0.205     & 0.298            & 0.211      & 0.292  & 0.312    & 0.299  & 0.172 \\		
		\bottomrule
	\end{tabular}
	\caption{Runtime (in seconds) of Algorithm \ref{alg:DP} for solving $\ell_2$-GNIO problem \eqref{eq:quad} under different settings. Results are averaged over 10 simulations.}
	\label{table:l2w}
\end{small}
\end{table}

Similar to the cases in \eqref{prob:LPgip} and \eqref{prob:LPiso}, $\ell_2$-GNIO problem \eqref{eq:quad} can  be equivalently recast as quadratic programming. Again, we compare our dynamic programming algorithm with Gurobi for solving problem \eqref{eq:quad} and present the detailed results in Table \ref{table:l2dpgu}. In the table, $0.000$ indicates that the runtime is less than $0.001$. Note that for simplicity, in these tests, we fix the weights by setting $w_i = 1/2$, $i=1,\ldots, n$.
We can observe that for most of the test instances in this class, our dynamic programming algorithm is able to outperform
the highly powerful quadratic programming solver in Gurobi by a factor of about {470--5122} in terms of computation times.

\begin{table}
\begin{small}
	\centering
	\begin{tabular}{@{}ccccc@{}}
		\toprule
		{$n$}   & parameters pattern & $t_{\rm DP}$ & $t_{\rm Gurobi}$  & \multicolumn{1}{c}{$t_{\rm Gurobi}/t_{\rm DP}$}     \\ \midrule
		1e4 & Isotonic    & 0.000           & 0.09        & $>90$ \\
		1e5 & Isotonic    & 0.002           & 1.78        & 890.00 \\
		1e6 & Isotonic    & 0.020           &  21.91      & 1095.50 \\
		1e7 & Isotonic    & 0.206           &  *      & * \\
		1e4 & Nearly-isotonic    & 0.000       & 0.09     & $>90$   \\
		1e5 & Nearly-isotonic    & 0.003       & 1.43     & 476.67     \\
		1e6 & Nearly-isotonic    & 0.031   &  14.69       & 473.87    \\
		1e7 & Nearly-isotonic    & 0.311   &  *       & *    \\
		1e4 & Unimodal    & 0.000          & 0.08         & $>80$ \\
		1e5 & Unimodal    & 0.002          & 2.16         & 1080.00 \\
		1e6 & Unimodal    & 0.019          &  31.04       & 1633.68  \\
		1e7 & Unimodal    & 0.197          &  *       & *  \\
		1e4 & Fused     & 0.000            & 0.12         & $>120$ \\
		1e5 & Fused     & 0.003            & 3.25         & 1083.33 \\
		1e6 & Fused    & 0.032             & 35.63   	  & 1113.44    \\
		1e7 & Fused    & 0.319             & *   	  & *    \\
		1e4 & Uniform   & 0.000            & 0.15        & $>150$     \\
		1e5 & Uniform   & 0.003            & 4.20         & 1400.00      \\
		1e6 & Uniform    & 0.033           & 43.07     	  & 1305.16   \\
		1e7 & Uniform    & 0.334           & *     	  & *  \\
		1e4 & Gaussian   & 0.000           & 0.15         & $>150$  \\
		1e5 & Gaussian   & 0.003           & 3.89         & 1296.67     \\
		1e6 & Gaussian    & 0.031          & 38.49        & 1241.61   \\ 
		1e7 & Gaussian    & 0.309          & *        & *  \\		
		1e4 & Mixed        & 0.000            & 0.15         & $>150$    \\	
		1e5 & Mixed        & 0.002           & 5.46        & 2730.00   \\
		1e6 & Mixed    	  & 0.020            & 102.44       & 5122.00   \\
		1e7 & Mixed       & 0.198            & *            & *    \\
		\bottomrule
	\end{tabular}
	\caption{Runtime (in seconds) comparisons between our dynamic programming Algorithm \ref{alg:DP} and Gurobi for solving problem \eqref{eq:quad} under different settings. Results are averaged over 10 simulations. { The entry ``*" indicates that Gurobi reports ``out of memory''}. }
	\label{table:l2dpgu}
\end{small}
\end{table}

\subsection{DP algorithm versus Gurobi: Real data}
In this subsection, we test our algorithm with real data. In particular, we collect the input vector $y\in \Re^n$ from various open sources. The following {four} data sets are collected and tested:
\begin{enumerate}
	\item gold: gold price index per minute in SHFE during (2010-01-01 -- 2020-06-30) \cite{indexdata};
	\item sugar: sugar price index per minute in ZCE (2007-01-01 -- 2020-06-30) \cite{indexdata};
	\item aep: hourly estimated energy consumption at American Electric Power (2004 -- 2008) \cite{energydata};
	\item ni: hourly estimated energy consumption at Northern Illinois Hub (2004 -- 2008) \cite{energydata}.
\end{enumerate}
In these tests, we fix the positive weights with $w_i = 1$ in $\ell_1$-GNIO problem \eqref{eq:abs} and $w_i = 1/2$ in $\ell_2$-GNIO problem \eqref{eq:quad}, $i=1,\ldots, n$. Similar to the experiments conducted in the previous subsection, under different settings of parameters $\{\lambda_i\}_{i=1}^{n-1}$ and $\{\mu_i\}_{i=1}^{n-1}$, we test our algorithm against Gurobi.
The detailed comparisons for solving $\ell_1$-GNIO and $\ell_2$-GNIO problems are presented in Tables \ref{tabreall1}.
As one can observe, our algorithm is quite robust and is much more efficient than Gurobi for solving many of these instances. 
For large scale $\ell_1$-GNIO problems with data sets {\bf sugar} and {\bf gold}, our algorithm can be over 600 times faster than Gurobi. Specifically, when solving the $\ell_1$-GNIO problem with the largest data set {\bf gold} under the ``Unimodal'' setting, our algorithm is 80,000 times faster than Gurobi. Meanwhile, for $\ell_2$-GNIO problems, our algorithm can be up to {18,000} times faster than Gurobi.
These experiments with real data sets again confirm the robustness and the high efficiency of our algorithm for solving various GNIO problems.

{
\begin{table}
\begin{footnotesize}
	\centering
	\resizebox{.99\columnwidth}{!}{
	\begin{tabular}{ccc|ccc|ccc}
		\toprule
		&&&\multicolumn{3}{c|}{$\ell_1$-GNIO} & \multicolumn{3}{c}{$\ell_2$-GNIO}\\[2pt]\midrule
		problem  & \multicolumn{1}{c}{$n$ } & parameters pattern & \multicolumn{1}{|c}{$t_{\rm DP}$} & \multicolumn{1}{c}{ $t_{\rm Gurobi}$  } & \multicolumn{1}{c|}{$t_{\rm Gurobi}/t_{\rm DP}$} & \multicolumn{1}{|c}{$t_{\rm DP}$} & \multicolumn{1}{c}{ $t_{\rm Gurobi}$  } & \multicolumn{1}{c}{$t_{\rm Gurobi}/t_{\rm DP}$}\\ [2pt] \midrule
		\multirow{7}{*}{sugar} & \multirow{7}{*}{923025} 
		& Isotonic 			& 0.512  & 989.37 & 1932.36 	& 0.019 & 39.81  & 2095.26 \\
		
		& & Nearly-isotonic 				& 0.271  & 168.69 & 622.47 & 0.018  & 20.11 & 1117.22\\
		& & Unimodal 						& 0.291  & 470.56 & 1617.04 & 0.024  & 39.32 & 1638.33 \\
		& & Fused 				& 0.426  & 410.98 & 964.74 & 0.027  & 34.74  & 1286.67\\
		& & Uniform 			& 0.162  & 661.20 & 4081.48 & 0.031 & 44.51 & 1435.81 \\
		& & Gaussian 			& 0.214  & 799.23 & 3734.72 &0.028 &50.09 &1788.93\\
		& & Mixed 			    & 0.164  & 702.13 & 4281.28 &0.017 & 312.25 & 18367.64 \\
		\hline     
		\multirow{7}{*}{gold} & \multirow{7}{*}{1097955} 
		& Isotonic     	& 0.203  & 772.91 & 3807.44 &0.021 & 38.16 &1817.14 \\ 
		& & Nearly-isotonic 				& 0.234  & 187.62 & 801.79 & 0.029 & 22.78 & 785.52 \\
		& & Unimodal 						& 0.235  & 19593.87 & 83378.20 &0.021 & 40.13 & 1910.95 \\
		& & Fused 			& 0.198  & 384.32 & 1941.01 & 0.027 &71.01 &2630.01 \\
		&                          & Uniform 			& 0.067  & 586.71 & 8756.87 &0.039 &51.66 &1324.62 \\
		&						     & Gaussian 			& 0.102  & 584.66 & 5731.96 &0.035 &75.07 &2144.86\\
		& & Mixed 			    & 0.054  & 782.46 & 14490.00 & 0.021 & 124.62 & 5934.28 \\
		\hline                        
		
		\multirow{7}{*}{aep} & \multirow{7}{*}{121273} 	                    		
		& Isotonic 	& 0.037  & 44.77 & 1210.01 &0.002 &5.12 &2560 \\ 
		& & Nearly-isotonic 				& 0.038  & 6.61  & 173.95 &0.003 &2.32 &733.33\\
		& & Unimodal 						& 0.040  & 31.57 &  789.25 &0.002 & 4.10 &2050.00 \\
		& & Fused     	& 0.037  & 19.15 & 517.56 &0.003 & 4.84 & 1613.33 \\
		&                         	 & Uniform 		& 0.021  & 18.47 & 879.52 &0.003 &5.05 &1683.33\\
		&							 & Gaussian 		& 0.037  & 29.02 & 784.32 &0.001 &4.96 &4960\\
		& & Mixed 			    & 0.020  & 14.74 & 737.00 &0.002 &35.50 & 17750.00\\
		\hline                            
		\multirow{7}{*}{ni} & \multirow{7}{*}{58450}                    	      
		& Isotonic 		& 0.018  & 14.55 & 808.33 &0.002 &1.78 & 895.00\\ 
		& & Nearly-isotonic 				& 0.018  & 2.76 & 153.33 & 0.001 &1.17 & 1170.00\\
		& & Unimodal 						& 0.019  & 10.47 & 551.05 &0.001 & 1.58 &1580.00\\
		&	& Fused      	& 0.018  & 5.14  & 285.56 & 0.001 & 2.21 & 2210.00\\
		&                            & Uniform 		& 0.011  & 8.21  & 746.36 &0.001 &5.36 &2680.00 \\
		&						    & Gaussian 		& 0.017  & 10.36 & 609.41 &0.001 &2.31 &2310.00\\
		& & Mixed 			    & 0.010  & 6.64 & 6640.00 & 0.001 &17.59 & 17590.00\\
		\bottomrule
	\end{tabular}%
	}
	\caption{Runtime (in seconds) comparisons between our dynamic programming Algorithm \ref{alg:DP} and Gurobi for solving $\ell_1$-GNIO problem \eqref{eq:abs} and $\ell_2$-GNIO problem \eqref{eq:quad} with real data.}
	\label{tabreall1}
\end{footnotesize}
\end{table}%
}

\subsection{DP algorithm versus Condat's direct algorithm}

To further evaluate the performance of our algorithm, we test our DP algorithm against Condat's direct algorithm \cite{Condat2013direct} for solving the following  fused lasso problem:
\begin{equation*}
\min_{x\in \Re^n}  \; \frac{1}{2} \sum_{i=1}^{n} (x_i - y_i)^2 + \lambda \sum_{i=1}^{n-1}|x_i - x_{i+1}|,
\end{equation*} 
where $\left\{ y_i \right\}_{i=1}^n$ are given data and $\lambda$ is a given positive regularization parameter.
As is mentioned in the beginning of the numerical section, {this} direct algorithm along with its highly optimized C implementation is regarded as one of the fastest algorithms for solving the above fused lasso problem \cite{Barbero2018modular}. %

In this subsection, we test five choices of $\lambda$, i.e., $\lambda = 1,2,5,10, 100$. The test data sets are the aforementioned four real data sets and two simulated random data sets of sizes $10^6$ and $10^7$. Table \ref{tabcondat} reports the detailed numerical results of the comparison between our algorithm and Condat's direct algorithm. As one can observe, in 14 out of 20 cases, our algorithm outperforms Condat's direct algorithm by a factor of about 1-2 in terms of computation times. These experiments indicate that for the fused lasso problem, our algorithm is highly competitive even when compared with the existing fastest algorithm in the literature. 

{
	\begin{table}
	\begin{small}
		\centering
		\begin{tabular}{cccccc}
			\toprule
			problem  & \multicolumn{1}{c}{$n$ } & $\lambda$ & \multicolumn{1}{c}{$t_{\rm DP}$} & \multicolumn{1}{c}{ $t_{\rm condat}$  } & \multicolumn{1}{c}{ $t_{\rm condat}/t_{\rm DP}$  }  \\ [2pt] \midrule
			\multirow{5}{*}{sugar} & \multirow{5}{*}{923025}        & 1 	& 0.030 & 0.038   & 1.27  \\ 
			& 						& 2 					               	& 0.030  & 0.034  & 1.13  \\
			& 							& 5			  						& 0.029  & 0.031  & 1.07  \\
			&                        & 10								    & 0.028  & 0.028  & 1.00  \\
			&                         & 100					                & 0.026  & 0.022  & 0.84 \\
			\hline    
			\multirow{5}{*}{gold} & \multirow{5}{*}{1097955}             & 1    	& 0.030  & 0.028 & 0.93 \\ 
			& 						& 2 						& 0.029  & 0.025 & 0.86 \\
			& 							& 5			  						& 0.028  & 0.023 &0.82 \\
			&                        & 10								& 0.028  & 0.022 & 0.79\\
			&                         & 100					& 0.026  & 0.019  &  0.73\\
			\hline                        		
			\multirow{5}{*}{aep} & \multirow{5}{*}{121273} 	                     		 & 1 	& 0.003  & 0.006 & 2.00 \\ 
			& 						& 2 						& 0.003  & 0.005 & 1.67\\
			& 							& 5			  						& 0.003  & 0.005 & 1.67\\
			&                        & 10								& 0.003  & 0.005 & 1.67\\
			&                         & 100					& 0.003  & 0.005  & 1.67 \\
			\hline                      
			\multirow{5}{*}{ni} & \multirow{5}{*}{58450} 	                  	        & 1 		& 0.001  & 0.002 & 2.00 \\
			& 						& 2 						& 0.001 & 0.002 & 2.00\\
			& 							& 5			  						& 0.001  & 0.002 & 2.00\\
			&                        & 10								& 0.001  & 0.002 & 2.00\\
			&                         & 100					& 0.001  & 0.002  & 2.00 \\
			\hline                            
			\multirow{5}{*}{random1} & \multirow{5}{*}{$10^6$} 	                  	        & 1 		& 0.032  & 0.042 & 1.31 \\
			& 						& 2 						& 0.031  & 0.042 & 1.35 \\
			& 							& 5			  						& 0.031  & 0.041 & 1.32 \\
			&                        & 10								& 0.030  & 0.041 & 1.37 \\
			&                         	& 100					& 0.026  & 0.035  & 1.35 \\
			\hline                            
			\multirow{5}{*}{random2} & \multirow{5}{*}{$10^7$} 	                  	        & 1 		& 0.286  & 0.370 & 1.29 \\
			& 						& 2 						& 0.279  & 0.375 & 1.34 \\
			& 							& 5			  						& 0.303  & 0.361 & 1.19 \\
			&                        & 10								& 0.299  & 0.362 & 1.21 \\
			&                         	& 100					& 0.296  & 0.282  & 0.95 \\
			\bottomrule
		\end{tabular}%
		\caption{Runtime (in seconds) comparisons between our dynamic programming Algorithm \ref{alg:DP} and {Condat's} algorithm \cite{Condat2013direct} for solving $\ell_2$-total variation problem \eqref{prob:flsa} with simulated and real data sets.}
		\label{tabcondat}%
	\end{small}
	\end{table}	
}

\section{Conclusions}
\label{sec:conclusion}
In this paper, we studied the generalized nearly isotonic optimization problems. The intrinsic recursion structure in these problems inspires us to solve them via the dynamic programming approach. By leveraging on the special structures of the generalized nearly isotonic regularizer in the objectives, we are able to show the computational feasibility and efficiency of the recursive minimization steps in the dynamic programming algorithm. Specifically, easy-to-implement and explicit updating formulas are derived for these steps. Implementation details, together with the optimal ${\cal O}(n)$ running time analysis, of our algorithm for solving $\ell_2$-GNIO problems are provided. Building upon all the aforementioned desirable results, our dynamic programming algorithm has demonstrated a clear computational advantage in solving large-scale $\ell_1$-GNIO and $\ell_2$-GNIO problems in
the numerical experiments when tested against the powerful commercial linear and quadratic programming solver Gurobi and other existing fast algorithms.

\begin{acknowledgement}
The authors would like to thank the Associate Editor and
anonymous referees for their helpful suggestions. 
\end{acknowledgement}

\section{Appendix}
\subsection{${\cal O}(n\log n)$ complexity of Algorithm \ref{alg:DP} for $\ell_1$-GNIO}
As is mentioned in Remark \ref{rmk:l1rmk}, when Algorithm \ref{alg:DP} is applied to solve $\ell_1$-GNIO, special data structures are needed to obtain a low-complexity implementation. Here, we propose to use a search tree to reduce the computation costs. The desired search tree should have following methods:
	\begin{itemize}
		\item \textit{insert}: adds a value into the tree while maintains the structure;
		\item \textit{popmax}: deletes and returns maximal value in the tree;
		\item \textit{popmin}: deletes and returns minimal value in the tree.
	\end{itemize}	
	These methods of the search tree are assumed to take ${\cal O}(\log n)$ time. We note that these requirements are not restrictive at all. In fact, the well-known red-black trees \cite{Cormen2009intro} satisfy all the mentioned properties. Now, we are able to sketch a proof for the ${\cal O}(n \log n)$ complexity of Algorithm \ref{alg:DP} for solving the $\ell_1$-GNIO. Using the desired search tree, in the ``sum'' step, the cost of insert a new breakpoint $y_i$  into the tree is $O(\log n)$. Note that at most $n$ breakpoints will be inserted. Hence, the  overall running time of the ``sum'' step is $O(n \log n)$.
	We also note that, in the ``update'' step, breakpoints will be deleted one by one via the \textit{popmax} or \textit{popmin} operations. {Recall that  the total number of  the deleted breakpoints in Algorithm \ref{alg:DP} is upper bounded by $2n$ (see the discussions in Section \ref{sec:imple_quadratic}).
	Therefore, the overall running time of the ``update'' step is  ${\cal O}(n \log n)$.} Hence, the total running time of Algorithm \ref{alg:DP} for solving $\ell_1$-GNIO is ${\cal O}(n \log n)$.


\begin{thebibliography}{}
	
\bibitem{Ahuja2001fast}
{R. K. Ahuja and J. B. Orlin}, {A fast scaling algorithm for minimizing separable convex functions subject to chain constraints}, Operations Research, 49, 784-789 (2001)

\bibitem{Ayer1955emprical}
{M. Ayer, H. D. Brunk, G. M. Ewing, W. T. Reid, and E. Silverman}, {An empirical distribution function for sampling with incomplete information}, Annals of Mathematical Statistics, 26, 641-647 (1955)

\bibitem{Barbero2018modular}
{\'{A}. Barbero and S. Sra}, {Modular proximal optimization for multidimensional total-variation regularization}, Journal of Machine Learning Research, 19, 2232-2313 (2018)

\bibitem{Bartholomew1959bio}
{D. J. Bartholomew}, {A test of homogeneity for ordered alternatives}, {Biometrika}, 46, 36-48 (1959)
	
\bibitem{bartholomew1959test}
{D. J. Bartholomew}, {A test of homogeneity for ordered alternatives II}, {Biometrika}, 46, 328-335 (1959)
	
\bibitem{Bartholomew1972statistic}
{R.~E. Barlow, D.~J. Bartholomew, J.~M. Bremner, and H.~D. Brunk}, {Statistical Inference under Order Restrictions: The Theory and
Application of Isotonic Regression}, Wiley, New York (1972)
	
\bibitem{Best1990active}
{M. J. Best and N. Chakravarti}, {Active set algorithms for isotonic regression; A unifying framework}, Mathematical Programming, 47, 425-439 (1990)
	
\bibitem{Best2000minimizing}
{M. J. Best, N. Chakravarti, and V. A. Ubhaya}, {Minimizing separable convex functions subject to simple chain constraints}, SIAM Journal on Optimization, 10, 658-672 (2000)
	
\bibitem{Brunk1955maximum}	
{H. D. Brunk}, {Maximum likelihood estimates of monotone parameters}, Annals of Mathematical Statistics, 26, 607-616 (1955)
	
	
\bibitem{xiaojun2016videocut}
{X. Chang, Y. Yu, Y. Yang, and E.~P. Xing}, {Semantic pooling for complex event analysis in untrimmed videos}, {IEEE Transactions on Pattern Analysis and Machine Intelligence}, 39, 1617-1732 (2016)


\bibitem{Cormen2009intro}
{T. H. Cormen, C. E. Leiserson, R. L. Rivest, and C. Stein}, {Introduction to Algorithms}, MIT Press (2009)
	
\bibitem{Condat2013direct}	
{L. Condat}, {A direct algorithm for 1D total variation denoising}, IEEE Signal Processing Letters, 20, 1054-1057 (2013)
	
\bibitem{Donoho1994ideal}
{D. L. Donoho and J. M. Johnstone}, {  Ideal spatial adaptation by wavelet shrinkage}, Biometrika, 81, 425-455 (1994)
	
\bibitem{Donoho1995denoising}
{D. L. Donoho}, {De-noising by soft-thresholding}, IEEE Transactions on Information Theory, 41, 613-627 (1995)
	
\bibitem{friedman2007pathwise}
{J. Friedman, T. Hastie, H. H{\"o}fling, and R. Tibshirani}, {Pathwise coordinate optimization}, {The Annals of Applied Statistics}, 1, 302-332 (2007)
	
\bibitem{Frisen1986unimodal}
{M. Frisen}, {Unimodal regression}, The Statistician, 35, 479-485 (1986)
	

\bibitem{hiriart2004fundamentals}
{J.-B. Hiriart-Urruty and C. Lemar{\'e}chal}, {Fundamentals of Convex Analysis}, Springer Science \& Business Media (2004)
	
\bibitem{Hochbaum2017faster}
{D.~S. Hochbaum and C. Lu}, {A faster algorithm solving a generalization of isotonic median regression and a class of fused lasso problems}, {SIAM Journal on Optimization}, 27, 2563-2596 (2017)
	
\bibitem{Hochbaum2001efficient}
{D. S. Hochbaum}, {An efficient algorithm for image segmentation, markov random fields and related problems}, Journal of the ACM, 48, 686-701 (2001)
	
\bibitem{Hochbaum2021unified}
{C. Lu and D. S. Hochbaum}, {A unified approach for a 1D generalized total variation problem}, Mathematical Programming (2021)
	
\bibitem{Hoefling2009path}
{H. H\"{o}efling}, {A path algorithm for the fused lasso signal approximator}, {Journal of Computational and Graphical Statistics}, 19, 984-1006 (2010)
	
\bibitem{Johnson2013dplasso}
{N.~A. Johnson}, {A dynamic programming algorithm for the fused lasso and $l_0$-segmentation}, {Journal of Computational and Graphical Statistics}, 22, 246-260 (2013)
	
\bibitem{indexdata}
{JoinQuant dataset}, {Continuous sugar price index `SR8888.XZCE' and continuous gold price index `AU8888.XSGE'}, \url{https://www.joinquant.com/data}
	
\bibitem{Matyasovszky2013climate}
{I. Matyasovszky}, {Estimating red noise spectra of climatological time series}, {Quarterly Journal of the Hungarian Meteorological Service}, 117, 187-200 (2013)
	
\bibitem{energydata}
{R. Mulla}, {Over 10 years of hourly energy consumption data from PJM in Megawatts}, \url{ https://www.kaggle.com/robikscube/hourly-energy-consumption?select=AEP\_hourly.csv}, Version 3
	
\bibitem{Alfred1993localiso}
{A. Restrepo and A.~C. Bovik}, {Locally monotonic regression}, {IEEE Transactions on Signal Processing}, 41, 2796-2810 (1993)
	
\bibitem{rockafellar1970convex}
{R.~T. Rockafellar}, {Convex Analysis}, Princeton University Press, Princeton, NJ (1970)
	
\bibitem{Rote2019isoDP}
{G. Rote}, {Isotonic regression by dynamic programming}, in {2nd Symposium on Simplicity in Algorithms} (2019)
		
\bibitem{Ryu2004medical}
{Y.~U. Ryu, R. Chandrasekaran, and V. Jacob}, {Prognosis using an isotonic prediction technique}, {Management Science}, 50, 777-785 (2004)
	
\bibitem{Silvapulle2005constrained}
{M. J. Silvapulle and P. K. Sen}, {Constrained Statistical Inference: Inequality, Order and Shape Restrictions}, John Wiley \& Sons (2005)
	
\bibitem{Quentin2008unimodal}
{Q.~F. Stout}, {Unimodal regression via prefix isotonic regression}, {Computational Statistics \& Data Analysis}, 53, 289-297 (2008)
	
	
\bibitem{stout2019fastest}
{Q.~F. Stout}, {Fastest known isotonic regression algorithms}, \url{https://web.eecs.umich.edu/~qstout/IsoRegAlg.pdf} (2019)
	
\bibitem{Stromberg1991algorithm}
{U. Str{\"o}mberg}, {An algorithm for isotonic regression with arbitrary convex distance function}, Computational Statistics \& Data Analysis, 11, 205-219 (1991) 
	
\bibitem{Tibshirani2016nearly}
{R. Tibshirani, H. H\"{o}efling, and R. Tibshirani}, {Nearly-isotonic regression}, {Technometrics}, 53, 54-61 (2011)
	
\bibitem{Tarjan1985amortized}
{R. E. Tarjan}, {Amortized computational complexity}, {SIAM Journal on Algebraic Discrete Methods}, 6, 306-318 (1985)
	
\bibitem{Wu2015traffic}
{C. Wu, J. Thai, S. Yadlowsky, A. Pozdnoukhov, and A. Bayen}, {Cellpath: Fusion of cellular and traffic sensor data for route flow estimation via convex optimization}, {Transportation Research Part C: Emerging Technologies}, 59, 111-128 (2015)
	
\bibitem{Yu2016exact}
{Y.-L. Yu and E. P. Xing}, {Exact algorithms for isotonic regression and related}, Journal of Physics: Conference Series 699 (2016)

\end{thebibliography}
\end{document}